\documentclass[10pt,reqno]{amsart}

\usepackage[utf8]{inputenc}
\usepackage[T1]{fontenc}
\usepackage{amsmath, amssymb, amsthm}
\usepackage{graphicx}
\usepackage[colorlinks=true,allcolors=blue]{hyperref}
\usepackage{geometry}
\usepackage{xcolor}
\usepackage{orcidlink}
\usepackage{setspace}
\usepackage{mathtools}

\DeclareMathOperator{\aut}{aut}
\DeclareMathOperator{\Aut}{Aut}
\DeclareMathOperator{\Hol}{Hol}
\DeclareMathOperator{\Iso}{Iso}

\DeclareMathOperator{\End}{End}
\DeclareMathOperator{\SO}{SO}
\DeclareMathOperator{\OO}{O}
\DeclareMathOperator{\Ad}{Ad}
\DeclareMathOperator{\ad}{ad}
\DeclareMathOperator{\Id}{Id}

\DeclareMathOperator{\trace}{trace}

\newcommand{\R}{\mathbb{R}}

\newcommand{\got}{\mathfrak}

\title{On the Structure of $2$-Step Nilpotent Lorentzian Naturally Reductive Lie Groups}

\author{Brian Luporini}
\address{CONICET and Departamento de Matemática, ECEN-FCEIA, Universidad Nacional de Rosario, Av.\ Pellegrini 250, S2000BTP Rosario, Argentina.}
\email{\href{mailto:luporini@fceia.unr.edu.ar}{luporini@fceia.unr.edu.ar}}
\urladdr{\href{https://orcid.org/0009-0003-5935-0582}{0009-0003-5935-0582}}

\author{Silvio Reggiani}
\address{CONICET and Departamento de Matemática, ECEN-FCEIA, Universidad Nacional de Rosario, Av.\ Pellegrini 250, S2000BTP Rosario, Argentina.}
\email{\href{mailto:reggiani@fceia.unr.edu.ar}{reggiani@fceia.unr.edu.ar}}
\urladdr{\href{https://orcid.org/0000-0002-8549-5828}{0000-0002-8549-5828}}

\author{Francisco Vittone}
\address{CONICET and Departamento de Matemática, ECEN-FCEIA, Universidad Nacional de Rosario, Av.\ Pellegrini 250, S2000BTP Rosario, Argentina.}
\email{\href{mailto:vittone@fceia.unr.edu.ar}{vittone@fceia.unr.edu.ar}}
\urladdr{\href{https://orcid.org/0000-0002-3883-273X}{0000-0002-3883-273X}}

\date{\today}

\theoremstyle{plain}
\newtheorem{theorem}{Theorem}[section]
\newtheorem{lemma}[theorem]{Lemma}
\newtheorem{corollary}[theorem]{Corollary}
\newtheorem{proposition}[theorem]{Proposition}

\theoremstyle{definition}
\newtheorem{definition}[theorem]{Definition}
\newtheorem{example}[theorem]{Example}

\theoremstyle{remark}
\newtheorem{observation}[theorem]{Remark}

\numberwithin{equation}{section}

\begin{document}

\begin{abstract}
We study $2$-step nilpotent Lie groups with naturally reductive left-invariant Lorentzian metrics with respect to the presentation group $N\rtimes H^{\aut}$. Replacing the standard non-degenerate center assumption with the weaker condition that the commutator ideal be non-degenerate, we develop a framework that extends the representation-theoretic construction to the Lorentzian setting and covers both the non-degenerate and degenerate center cases. In the degenerate case, we show that the associated Lie algebra is a central extension of a semidirect product whose Riemannian factor is naturally reductive. Furthermore, we obtain invariant decompositions of the defining representation in the non-degenerate case. In the degenerate case, we introduce a Riemannian quotient representation and show that the kernel of this representation is a central abelian ideal whose elements act by null rotations. Finally, we use these results to characterize  the isotropy algebra of the group of isometric automorphisms.
\end{abstract}

\maketitle

\section{Introduction}\label{sec:1_Introduction}

The geometry of nilpotent Lie groups endowed with a left-invariant metric has been a recurring subject of investigation in homogeneous geometry, both in the Riemannian and the pseudo-Riemannian settings. Among these Lie groups, the naturally reductive ones play an important role. In this paper, we address the structure of $2$-step nilpotent Lie groups with a Lorentzian left-invariant metric that is naturally reductive with respect to a suitable presentation group.

The Riemannian case is by now classical. Wolf \cite{Wolf1962} showed that if a connected nilpotent Lie group $N\subset\Iso(M)$ acts transitively on a Riemannian manifold $M$, then $N$ is unique, coincides with the nilradical of the isometry group, and the action is simply transitive. Consequently, $M$ can be identified with $N$ together with a left-invariant metric, the isotropy subgroup $H$ at the identity equals the group $H^{\aut}$ of isometric automorphisms of $N$, and the full isometry group splits as the semidirect product
\begin{equation*}
\Iso(M)=N\rtimes H^{\aut}.
\end{equation*}
In this framework, Kaplan introduced and studied the so-called $H$-type Lie groups in \cite{Kaplan1981}, while Wilson \cite{Wilson1982} and Gordon \cite{Gordon1985} carried out a systematic analysis of homogeneous nilmanifolds. In particular, Gordon proved that a naturally reductive Riemannian nilpotent Lie group with a left-invariant metric must be at most $2$-step nilpotent. Later, Lauret described naturally reductive nilmanifolds via representations of compact Lie algebras and derived a number of remarkable geometric consequences \cite{Lauret1998,Lauret1999}.

The pseudo-Riemannian setting is considerably more delicate. Ovando \cite{Ovando2013} extended Lauret's construction and provided a description of naturally reductive $2$-step nilpotent Lie groups with a left-invariant pseudo-Riemannian metric, under the assumption that the center $\got z$ of the Lie algebra is non-degenerate. Already in this generality the picture changes substantially with respect to the Riemannian case: not every naturally reductive nilmanifold is $2$-step nilpotent, and, as shown by del Barco and Ovando \cite{DelBarcoOvando2014}, the subgroup $N\rtimes H^{\aut}$ described by Wolf is in general strictly smaller than the full isometry group $\Iso(N)$. Further results on the structure of pseudo-Riemannian $2$-step nilpotent Lie groups and their isometry groups can be found in \cite{CorderoParker2009, Guediri2003, ContiDelBarcoRossi2024}.

In the Lorentzian setting, a major breakthrough was obtained in the recent works of Chen, Wolf, Zhang and Nikolayevsky \cite{ChenWolfZhang2020,NikolayevskyWolf2022}: if $N$ is a simply connected nilpotent Lie group endowed with a left-invariant Lorentzian metric which is naturally reductive with respect to the subgroup $N\rtimes H^{\aut}$ of $\Iso(N)$, and if the commutator $\got n'=[\got n,\got n]$ is non-degenerate, then $N$ is necessarily $2$-step nilpotent, just as in the Riemannian case. When $\got n'$ is degenerate, by contrast, the step of nilpotency can be arbitrarily large. 

The aim of the present paper is to complete the structural picture of  nilpotent Lorentzian Lie groups that are naturally reductive
with respect to $N\rtimes H^{\aut}$. Motivated by the results of Nikolayevsky and Wolf, we replace the hypothesis used in \cite{Ovando2013} (non-degeneracy of the center $\got z$ and injectivity of $j$) with the weaker and geometrically more natural assumption that the commutator $\got n'$ is non-degenerate. Under this assumption, $\got n$  is always $2$-step nilpotent and we have the orthogonal decomposition
$\got n=\got n'\oplus\got v$. Also,  the Lie bracket determines an injective linear map
\begin{equation*}
j:\got n'\to\got{so}(\got v),
\qquad
\langle [X,Y],Z\rangle=\langle j(Z)X,Y\rangle.
\end{equation*}
After splitting off a flat factor, the analysis reduces to two
qualitatively different situations:
\begin{itemize}
    \item the case where the center $\got z$ is non-degenerate, in which case $\got z=\got n'$ and the construction extends the
    picture in \cite{Ovando2013};
    \item the case where $\got z$ is degenerate, in which case $\got z=\got n'\oplus\mathbb RE$ for a non-zero lightlike vector $E\in\got v$.
\end{itemize}
The latter situation is genuinely Lorentzian in nature and has no Riemannian counterpart. It requires new techniques and produces a
structural description that, to the best of our knowledge, has not appeared before in the literature.

Our main contributions can be summarized as follows. First, we establish a Lorentzian analogue of Lauret's representation-theoretic construction of naturally reductive nilmanifolds, which includes both the non-degenerate and the degenerate
behavior of the center. More precisely, we show that when $\got n'$ is non-degenerate, natural reductivity with respect to $N\rtimes H^{\aut}$ can be characterized in terms of the representation $j$. This leads naturally to the notions of a Lorentzian data set and a degenerate Lorentzian data set, which depend, respectively, on the non-degeneracy or the degeneracy of the center $\got z$ of $\got n$. In the first case the defining representation has trivial common kernel, whereas in the second case its common kernel is precisely a one-dimensional lightlike subspace. Conversely, the nilpotent Lie group associated with either type of data set is naturally reductive and does not split off a non-trivial flat factor.

We then analyze in more detail the structure arising in the degenerate case. We show that any $2$-step nilpotent Lie algebra $\got n(\got g,\got v,\pi)$ associated with a degenerate Lorentzian data set is a central extension of a semidirect product whose
underlying factor is the Lie algebra of a Riemannian naturally reductive nilpotent Lie group. We also provide a converse construction under a natural condition on the corresponding derivation. In this way, part of the Lorentzian structure can be reduced to data already understood in the Riemannian setting. This construction is similar to that of \cite{NikolayevskyWolf2022} for those g.o. nilpotent Lie groups $N$ such that $\got n'$ is degenerate. In fact, \cite[Theorem 3]{NikolayevskyWolf2022} shows that, when \(\got n'\) is degenerate, $\got n$ is a double extension of a Riemannian geodesic orbit nilpotent Lie algebra. In the naturally reductive setting considered here, the argument can be adapted to show that the Riemannian factor is naturally reductive as well; see Remark \ref{rem:NWOLF}.

A second goal of the paper is to describe the representations $\pi:\got g\to\got{so}(\got v)$ arising from these data sets. We first show that $\got g$ is a compact Lie algebra and then study separately the Lorentzian and degenerate cases. When the center $\got z$ of the resulting nilpotent Lie algebra is non-degenerate, we obtain an orthogonal decomposition of $\got v$ into invariant subspaces, generalizing to the Lorentzian case the corresponding decompositions in \cite{Lauret1999,Ovando2013}. In particular, the action of the center of $\got g$ on the irreducible Riemannian factors is described by orthogonal complex structures, while the Lorentzian part contains a distinguished $2$-dimensional factor spanned by invariant lightlike vectors.

In the degenerate case, the existence of the common lightlike kernel $\mathbb RE$ prevents one from applying the same decomposition
directly. We therefore consider the quotient $\overline{\got v}=(\mathbb RE)^\perp/\mathbb RE$,  which is Riemannian, together with the induced representation $\overline\pi:\got g\to\got{so}(\overline{\got v})$. We show that $\got k=\ker\overline\pi$
is an abelian ideal contained in the center of $\got g$. Moreover, after choosing a lightlike vector $F$ such that
$\langle E,F\rangle=1$, the elements of $\got k$ act explicitly as null rotations. This quotient representation allows us to recover a Riemannian representation-theoretic structure while keeping track of the additional Lorentzian information contained in $\got k$.

Finally, we use these decompositions to give an explicit description of the isotropy algebra $\got h^{\aut} = \mathrm{Der}(\got n)\cap\got{so}(\got n)$ of $N\rtimes H^{\aut}$. The corresponding result in \cite{Ovando2013} relies on the additional hypothesis that $\got g$ is semisimple; in the present Lorentzian setting we show that it is never the case. The center of $\got g$ is always non-trivial and plays an essential role.

The paper is organized as follows. In Section \ref{sec:2_Geometry_of_2-step_nilpotent_Lorentzian_Lie_groups}, we collect the
geometric preliminaries on $2$-step nilpotent Lorentzian Lie groups, derive the Lorentzian analogue of Eberlein's description of the
Levi-Civita connection \cite{Eberlein1994}, and discuss the splitting of a flat factor. Section \ref{sec:3_Pseudo-Riemannian_naturally_reductive_spaces} reviews the general framework of pseudo-Riemannian naturally reductive spaces, introduces the
presentation group
\begin{equation*}
    \Iso^{\aut}(N)=N\rtimes H^{\aut},
\end{equation*}
characterizes it as the normalizer of $N$ in $\Iso(N)$, and uses the Geodesic Lemma of \cite{ChenWolfZhang2020,NikolayevskyWolf2022} to characterize natural reductivity in terms of the representation $j$. In Section \ref{sec:4_Naturally_reductive_Lorentzian}, we develop the representation-theoretic construction, introducing the notions of Lorentzian data set and degenerate Lorentzian data set according to whether the center of the resulting Lie algebra is non-degenerate or degenerate. Section \ref{sec:5_Degenerate_data_set} is devoted to the degenerate case, where we describe the associated Lie algebra as a central extension of a semidirect product related to a Riemannian naturally reductive nilpotent Lie group. We also analyze briefly a similar construction when $\got n'$ is degenerate.  In Section \ref{sec:6_Action_of_the_representation_pi_on_v} we analyze the action of $\pi$ on $\got v$ and obtain the corresponding invariant decompositions. Finally, in Section \ref{sec:7_action_degenerate}, we use these results to describe the isotropy algebra $\got h^{\aut}$ in both the non-degenerate and degenerate cases.

\section{Geometry of \texorpdfstring{$2$}{2}-step nilpotent Lorentzian Lie groups}\label{sec:2_Geometry_of_2-step_nilpotent_Lorentzian_Lie_groups}

In this section, we obtain some results on the geometry of simply connected, Lorentzian, $2$-step nilpotent Lie groups. As mentioned in the introduction, we assume that the commutator of the Lie algebra of such a group is non-degenerate, and we obtain a description of the Levi-Civita connection analogous to those presented in~\cite{Ovando2013} and~\cite{Eberlein1994}.

Let $N$ be a simply connected, $2$-step nilpotent Lie group endowed with a pseudo-Riemannian left-invariant metric $\langle \cdot,\cdot\rangle$. Let $\got n$ be its Lie algebra (i.e., the Lie algebra of left-invariant vector fields of $N$). We denote by $\got n' = [\got n,\got n]$ the commutator of $\got n$ and by $\got z$ its center. Since $\got n$ is $2$-step nilpotent, $\got n' \subseteq \got z$.

Assume that $\got n' = [\got n,\got n]$ is a non-degenerate subspace of $\got n$, and consider the orthogonal decomposition
\begin{equation}\label{eq:descortn}
    \got n = \got n' \oplus \got v
\end{equation}
with $\got v = (\got n')^{\perp}$. Since $[X,Y] \in \got n'$ for all $X,Y \in \got n$, one can define a linear map $j: \got n' \to \End(\got v)$ such that
\begin{equation}\label{eq:def_map_j}
    \langle [X, Y], Z \rangle = \langle j(Z)X, Y \rangle, \quad \text{for } X, Y \in \got v, \; Z \in \got n'.
\end{equation}
Note that for each $X \in \got v$, $\ad_X|_{\got v}: \got v \to \got n'$ and so $j(Z)(X) = (\ad_X|_{\got v})^{*}$, where $*$ denotes the adjoint map.

It is immediate from~\eqref{eq:def_map_j} that $j(Z) \in \got{so}(\got v, \langle \cdot, \cdot \rangle|_{\got v})$. Observe now that if $j(Z) = 0$, then $\langle [X,Y], Z \rangle = 0$ for all $X,Y \in \got v$. Since $\got n' = [\got v, \got v]$, this cannot happen, and so $j: \got n' \to \got{so}(\got v, \langle \cdot, \cdot \rangle|_{\got v})$ is injective.

\begin{lemma}\label{lemma:nabla_in_2stepnilpo}
    Let $N$ be a simply connected $2$-step nilpotent Lie group with Lie algebra $\got n$, endowed with a left-invariant pseudo-Riemannian metric such that $\got n'$ is non-degenerate. Decompose $\got n$ as in~\eqref{eq:descortn} and define $j$ as in~\eqref{eq:def_map_j}. If $Z, Z' \in \got n'$ and $X, Y \in \got v$, then
    \begin{equation*}
        \left\{
        \begin{array}{l}
            \nabla_{Z} Z' = 0, \\[2pt]
            \nabla_X Z = \nabla_Z X = -\tfrac{1}{2} j(Z) X, \\[2pt]
            \nabla_X Y = \tfrac{1}{2} [X,Y].
        \end{array}
        \right.
    \end{equation*}
\end{lemma}

\begin{proof}
    Let $A, B, C \in \got n$. Since the metric on $N$ is left-invariant, the Koszul formula gives
    \begin{equation}\label{eq:koszul_proof}
        2 \langle \nabla_A B, C \rangle = \langle [A,B], C \rangle - \langle [A,C], B \rangle - \langle [B,C], A \rangle.
    \end{equation}

    Since $\got n' \subseteq \got z$, if we take $A = X \in \got v$ and $B = Z \in \got n'$, then for each $C \in \got n$,
    \begin{equation*}
        2 \langle \nabla_X Z, C \rangle = -\langle [X,C], Z \rangle.
    \end{equation*}
    Write $C = C_1 + C_2$, with $C_1 \in \got n'$ and $C_2 \in \got v$. Then $[X,C] = [X, C_2]$ and
    \begin{equation*}
        \langle [X,C], Z \rangle = \langle [X, C_2], Z \rangle = \langle j(Z) X, C_2 \rangle = \langle j(Z) X, C \rangle,
    \end{equation*}
    since $C_1 \perp j(Z) X$.

    From the Koszul formula~\eqref{eq:koszul_proof} it is immediate that if $A = Z' \in \got n'$, then $\nabla_{Z'} Z = 0$. Hence, if $Z \in \got n'$, then
    \begin{equation*}
        \nabla_A Z = \left\{
        \begin{array}{cl}
            0 & \text{if } A \in \got n', \\[2pt]
            -\tfrac{1}{2} j(Z) A & \text{if } A \in \got v.
        \end{array}
        \right.
    \end{equation*}
    Moreover, if $Z \in \got n'$ and $A \in \got n$, then
    \begin{equation*}
        \nabla_Z A = \nabla_A Z + [Z,A] = \nabla_A Z.
    \end{equation*}

    Now take $A = X$, $B = Y \in \got v$ in~\eqref{eq:koszul_proof}. Since $\got v \perp \got n'$, for each $C \in \got n$,
    \begin{equation*}
        2 \langle \nabla_X Y, C \rangle = \langle [X,Y], C \rangle,
    \end{equation*}
    so $\nabla_X Y = \tfrac{1}{2} [X,Y]$.
\end{proof}

\begin{lemma}\label{lemma:ideala0}
    Let $N$ be a simply connected $2$-step nilpotent Lie group with Lie algebra $\got n$, endowed with a left-invariant pseudo-Riemannian metric such that $\got n'$ is non-degenerate. Decompose $\got n$ as in~\eqref{eq:descortn} and define $j$ as in~\eqref{eq:def_map_j}. Let
    \begin{equation*}
        \got a_0 = \bigcap_{Z \in \got n'} \ker j(Z) \subseteq \got v.
    \end{equation*}
    Then:
    \begin{enumerate}
        \item $\got a_0$ is an abelian ideal of $\got n$ such that $\got a_0 \subseteq \got z$ and $\got z = \got n' \oplus \got a_0$.
        \item Every left-invariant vector field in $\got a_0$ is parallel; that is, if $X \in \got a_0$, then $\nabla X = 0$.
        \item Let $\nu_p$ be the nullity subspace of $N$ at $p \in N$, i.e.,
        \begin{equation*}
            \nu_p = \{ v \in T_p N : R(\cdot, \cdot) v \equiv 0 \},
        \end{equation*}
        where $R$ is the curvature tensor of $N$. Then $\got a_0 \cdot p := \{ X_p : X \in \got a_0 \} \subseteq \nu_p$.
        \item If the metric on $N$ is Lorentzian, then $\got a_0 \cdot p = \nu_p$.
    \end{enumerate}
\end{lemma}

\begin{proof}
    Let $X \in \got a_0$ and $Y$ be any element of $\got v$. Then for each $Z \in \got n'$,
    \begin{equation*}
        \langle [X,Y], Z \rangle = \langle j(Z) X, Y \rangle = 0.
    \end{equation*}
    Since $\got n'$ is non-degenerate, we conclude that $[X,Y] = 0$ for each $Y \in \got v$, and since $\got n' \subseteq \got z$, we also have $[X,Z] = 0$ for each $Z \in \got n'$. Thus $X \in \got z$, so $\got a_0 \subseteq \got z$. Hence $\got a_0$ is an abelian ideal of $\got n$.

    Now let $A \in \got z$ and write $A = Z_A + X_A$ with $Z_A \in \got n'$ and $X_A \in \got v$. Then for each $Y \in \got v$ and each $Z \in \got n'$,
    \begin{equation*}
        0 = \langle [A,Y], Z \rangle = \langle [X_A, Y], Z \rangle = \langle j(Z) X_A, Y \rangle.
    \end{equation*}
    We conclude that $j(Z) X_A = 0$ for each $Z \in \got n'$, and so $X_A \in \got a_0$.

    Since $\got a_0 \subset \got v$, we have $\got n' \cap \got a_0 = \{0\}$. Therefore
    \begin{equation*}
        \got z = \got n' \oplus \got a_0.
    \end{equation*}

    The second assertion is immediate from Lemma~\ref{lemma:nabla_in_2stepnilpo}. 

    To prove the last two assertions, observe first that, from the second assertion, $\got a_0\cdot p\subseteq \nu_p$ for all $p\in N$.

    Assume now that the metric on $N$ is Lorentzian. Since $\got a_0$ and $\nu$ are both left-invariant, it suffices to prove that $\got a_0\cdot e=\nu_e$, where $e$ is the identity element of $N$. So let $w\in\nu_e$ and let $W\in\got n$ be the left-invariant vector field defined by $w$. Decompose $W=Z_W+X_W$ with $Z_W\in\got n'$ and $X_W\in\got v$. Then for each $A,B\in\got n$,
    \begin{equation*}
        0=R(A,B)W=R(A,B)Z_W+R(A,B)X_W.
    \end{equation*}

    Take $Y\in\got v$ and $Z\in\got n'$. Then $\nabla_{[Y,Z]}Z_W=\nabla_{[Y,Z]}X_W=0$, and from Lemma \ref{lemma:nabla_in_2stepnilpo} we have
    \begin{align*}
        R(Y,Z)Z_W
        &=\nabla_Y\nabla_ZZ_W-\nabla_Z\nabla_YZ_W
          =-\nabla_Z\nabla_YZ_W\\
        &=\tfrac12\nabla_Zj(Z_W)Y
          =-\tfrac14j(Z)\circ j(Z_W)Y\in\got v\\
        R(Y,Z)X_W
        &=\nabla_Y\nabla_ZX_W-\nabla_Z\nabla_YX_W\\
        &=-\tfrac12\nabla_Yj(Z)X_W
          -\tfrac12\nabla_Z[Y,X_W]\\
        &=-\tfrac14[Y,j(Z)X_W]\in\got n'.
    \end{align*}

    So for each $Y\in\got v$ and $Z\in\got n'$,
    $R(Y,Z)W=0$ if and only if
    \begin{equation*}
        R(Y,Z)Z_W=R(Y,Z)X_W=0.
    \end{equation*}

    Since $R(Y,Z)Z_W=0$, we have
    \begin{equation*}
        j(Z)(j(Z_W)Y)=0
    \end{equation*}
    for each $Z\in\got n'$, and so $j(Z_W)Y\in\got a_0$ for each $Y\in\got v$. Take an arbitrary $X\in\got a_0$. Then
    \begin{equation*}
        \langle j(Z_W)Y,X\rangle
        =-\langle Y,j(Z_W)X\rangle=0.
    \end{equation*}
    Hence
    \begin{equation*}
        \operatorname{Im}j(Z_W)
        \subseteq\got a_0\cap\got a_0^\perp
        =\operatorname{rad}(\got a_0).
    \end{equation*}

    Since the metric on $\got n$ is Lorentzian, either $\got a_0$ is non-degenerate or $\dim\operatorname{rad}(\got a_0)=1.$

    In the first case, $j(Z_W)=0$. In the second case, $\operatorname{rank}j(Z_W)\leq1$. But the rank of a skew-symmetric endomorphism of a non-degenerate scalar product space is even. Therefore $j(Z_W)=0$ also in this case. Since $j$ is injective, $Z_W=0$.

    Thus $W=X_W\in\got v$. Now,
    \begin{equation*}
        [Y,j(Z)X_W]=-4R(Y,Z)X_W=0
    \end{equation*}
    for each $Y\in\got v$ and $Z\in\got n'$. Then $j(Z)X_W\in\got z$ for every $Z\in\got n'$, and $j(Z)X_W\in\got v\perp\got n'$. Hence $j(Z)X_W\in\got a_0$. Moreover, if $X\in\got a_0$, $ \langle j(Z)X_W,X\rangle = -\langle X_W,j(Z)X\rangle = 0$. 
    Therefore 
    \begin{equation*}
        j(Z)X_W\in\operatorname{rad}(\got a_0)
    \end{equation*}
    for every $Z\in\got n'$.

    If $\got a_0$ is non-degenerate, it follows immediately that $j(Z)X_W=0$ for every $Z\in\got n'$, and hence $X_W\in\got a_0$.

    Suppose then that $\got a_0$ is degenerate. Since the metric on $\got n$ is Lorentzian, there exists a non-zero lightlike vector
    $E\in\got a_0$ such that $\operatorname{rad}(\got a_0)=\mathbb RE$. Thus, for every $Z\in\got n'$, there exists $\lambda(Z)\in\mathbb R$ such that
    \begin{equation*}
        j(Z)X_W=\lambda(Z)E.
    \end{equation*}
    The map $\lambda:\got n'\to\mathbb R$ is linear. Since $\got a_0$ is degenerate and $\got a_0\subseteq\got v$, then $\got v$ is Lorentzian and consequently, the restriction of the metric to $\got n'$ is positive-definite. So there exists a unique $L\in\got n'$ such that for every $Z\in \got n'$
    \begin{equation*}
        \lambda(Z)=\langle L,Z\rangle.
    \end{equation*}
    From the definition of $j$ we obtain, for every $Y\in\got v$ and $Z\in\got n'$,
    \begin{equation*}
        \langle[X_W,Y],Z\rangle
        =\langle j(Z)X_W,Y\rangle
        =\lambda(Z)\langle E,Y\rangle
        =\langle L,Z\rangle\langle E,Y\rangle.
    \end{equation*}
    Since the metric on $\got n'$ is non-degenerate, $[X_W,Y]=\langle E,Y\rangle L.$

    Choose $F\in\got v$ such that $\langle E,F\rangle=1$. Then $[X_W,F]=L$. Since $W=X_W\in\nu_e$, we have $ R(X_W,F)X_W=0$ and Lemma~\ref{lemma:nabla_in_2stepnilpo} gives
    \begin{align*}
        R(X_W,F)X_W
        &=\nabla_{X_W}\nabla_FX_W
          -\nabla_F\nabla_{X_W}X_W
          -\nabla_{[X_W,F]}X_W\\
        &=\nabla_{X_W}\left(-\tfrac12L\right)
          -\nabla_LX_W=\tfrac14j(L)X_W+\tfrac12j(L)X_W\\
        &=\tfrac34j(L)X_W.
    \end{align*}
    Hence $j(L)X_W=0$.  But $j(L)X_W =\lambda(L)E =\langle L,L\rangle E$. Since the metric on $\got n'$ is positive-definite and $E\neq0$, we conclude that $L=0$. Therefore $\lambda=0$, and hence
    \begin{equation*}
        j(Z)X_W=0
    \end{equation*}
    for every $Z\in\got n'$. Thus $X_W\in\got a_0$.
\end{proof}

\begin{observation}
    If $\got z$ is non-degenerate, one can define $J: \got z \to \got{so}(\got z^{\perp})$ in the same way as in~\eqref{eq:def_map_j} (cf.~\cite{Eberlein1994, Ovando2013}). With this definition, the ideal $\got a_0$ of Lemma~\ref{lemma:ideala0} coincides with $\ker(J)$ and, in the Riemannian case, corresponds to the flat factor of the Lie group $N$ (cf.~\cite[Prop.~2.7]{Eberlein1994}). We prove next a similar result in the pseudo-Riemannian setting.
\end{observation}

We say that a simply connected pseudo-Riemannian manifold $M$ \emph{splits off a flat factor} if $M$ is isometric to a product $M_0 \times M'$, where $M_0 \simeq \mathbb{R}^k$ is a pseudo-Riemannian flat space form and $M'$ is a pseudo-Riemannian manifold.

The following theorem summarizes the results of~\cite{Wu1964} on the decomposition of pseudo-Riemannian manifolds.

\begin{theorem}\label{theorem:Wu}
    Let $M$ be a geodesically complete simply connected pseudo-Riemannian manifold and let $p \in M$. Let $\Hol(M,p)$ denote the holonomy group of $M$ at $p$. Suppose that $T_p M$ orthogonally splits as $T_p M = V_1 \oplus V_2$, where $V_1$ and $V_2$ are non-degenerate $\Hol(M,p)$-invariant subspaces of $T_p M$, and let $\mathcal{V}_1$ and $\mathcal{V}_2$ be the distributions on $M$ defined by parallel translation of $V_1$ and $V_2$ along $M$. Then:
    \begin{enumerate}
        \item $\mathcal{V}_1$ and $\mathcal{V}_2$ are integrable distributions. Their maximal connected integral submanifolds $M_1$ and $M_2$ are simply connected, geodesically complete and totally geodesic submanifolds of $M$, and $M$ is isometric to the pseudo-Riemannian product $M_1 \times M_2$. Moreover, if $p = (p_1, p_2)$, then $\Hol(M,p)$ splits as $\Hol(M,p) \simeq \Hol(M_1, p_1) \times \Hol(M_2, p_2)$, where $\Hol(M_i, p_i)$ acts trivially on $T_{p_j}M_j$ for $i \neq j$.
        \item If $\Phi(v) = v$ for all $v \in V_1$ and all $\Phi \in \Hol(M,p)$, then $\Hol(M_1, p_1)$ is trivial and hence $M$ splits off a flat factor $M_1$.
        \item $M_1$ is the flat factor of $M$ if $V_1$ is the set of fixed points of $\Hol(M,p)$.
    \end{enumerate}
\end{theorem}

\begin{observation}
    Given a vector space $W$ with bilinear form $\langle \cdot, \cdot \rangle$, we say for simplicity that $W$ is \emph{Riemannian} (or that $\langle \cdot, \cdot \rangle$ is Riemannian) if $\langle \cdot, \cdot \rangle$ is positive-definite. Similarly, $W$ is \emph{Lorentzian} (or $\langle \cdot, \cdot \rangle$ is Lorentzian) if $\langle \cdot, \cdot \rangle$ has signature one, and $W$ is \emph{degenerate} if $\langle \cdot, \cdot \rangle$ is degenerate.
\end{observation}

As a consequence of Lemma~\ref{lemma:ideala0} and Theorem~\ref{theorem:Wu} we have:

\begin{proposition} \label{prop:not_split-off_a_flat_factor}
    Let $N$ be a simply connected, geodesically complete, $2$-step nilpotent Lie group with Lie algebra $\got n$, endowed with a left-invariant Lorentzian metric such that $\got n'=[\got n,\got n]$ is non-degenerate. Consider the orthogonal decomposition $\got n=\got n'\oplus\got v$ and let $\got a_0=\bigcap_{Z\in\got n'}\ker j(Z)$. Then  the following conditions are equivalent:
    \begin{enumerate}
        \item\label{item1Wu} $N$ does not split off a non-trivial flat factor;
    
        \item\label{item2Wu} $\got a_0$ contains no non-zero non-degenerate subspace;
    
        \item\label{item3Wu} either $\got z=\got n'$ or $\got z$ is degenerate and $\got z=\got n'\oplus\mathbb RE$ for some non-zero lightlike vector $E\in\got v$.
    \end{enumerate}
    In particular, condition {\rm(3)} is equivalent to saying that either $\got z=\got n'$, or $\got z$ is degenerate and $\dim\got z=\dim\got n'+1.$
\end{proposition}

\begin{proof}
    By Lemma \ref{lemma:ideala0}, $\got z=\got n'\oplus\got a_0$, every element of $\got a_0$ determines a parallel left-invariant vector field on $N$, and, since the metric on $N$ is Lorentzian, $\got a_0=\nu_e$.
    
    We first prove the equivalence between \eqref{item1Wu} and \eqref{item2Wu}. Suppose that $\got a_0$ contains a non-zero non-degenerate subspace $\got a_1$. Since every element of $\got a_0$ determines a parallel left-invariant vector field, $\got a_1$ is fixed pointwise by the holonomy group of $N$. Since $\got a_1$ is non-degenerate, Theorem~\ref{theorem:Wu} implies that $N$ splits off a non-trivial flat factor.

    Conversely, suppose that $N$ splits off a non-trivial flat factor, say $N\simeq N_0\times N_1$, where $N_0$ is a non-trivial flat pseudo-Riemannian manifold. The subspace $V_0=T_eN_0\subset T_eN\simeq\got n$ is non-zero and non-degenerate. Moreover, the holonomy group of $N$ acts trivially on $V_0$. Therefore $V_0\subseteq\nu_e=\got a_0$. Thus $\got a_0$ contains a non-zero non-degenerate subspace.

    It remains to prove the equivalence of \eqref{item2Wu} and \eqref{item3Wu}. Since $\got z=\got n'\oplus\got a_0$, $\got n'$ is non-degenerate and $\got a_0\subset\got v=(\got n')^\perp$, the sum is orthogonal.

    Suppose that \eqref{item2Wu} holds. Then either $\got a_0=\{0\}$ or $\got a_0$ is degenerate. Thus either $\got z=\got n'$ or $\got z=\got n'\oplus\got a_0$ with $\got a_0$ degenerate.

    If $\got a_0$ is degenerate, since the metric on $\got n$ is Lorentzian, the radical of $\got a_0$ is one-dimensional. Thus there exists a non-zero lightlike vector $E\in\got a_0$ and a Riemannian subspace $\got m_0$ such that $\got a_0=\got m_0\oplus\mathbb RE$ as an orthogonal direct sum. If $\got m_0\neq\{0\}$, then $\got m_0$ would be a non-zero non-degenerate subspace of $\got a_0$, contradicting \eqref{item2Wu}. Hence $\got a_0=\mathbb RE$, and therefore $\got z=\got n'\oplus\mathbb RE$.  Thus \eqref{item3Wu} holds.

    Conversely, if $\got z=\got n'$, Lemma~\ref{lemma:ideala0} implies that $\got a_0=\{0\}$, and hence \eqref{item2Wu} holds.
    If instead $\got z=\got n'\oplus\mathbb RE$ with $E$ lightlike, then Lemma~\ref{lemma:ideala0} implies that
    $\got a_0=\mathbb RE$. Hence $\got a_0$ contains no non-zero non-degenerate subspace, so again \eqref{item2Wu} holds.

    The final assertion follows immediately from Lemma~\ref{lemma:ideala0}.
\end{proof}

\begin{proposition}\label{prop:lightl_vect_skew-sym_der}
    Let $N$ be a simply connected $2$-step nilpotent Lie group with Lie algebra $\got n$, endowed with a left-invariant Lorentzian metric such that $\got n'$ is non-degenerate. Let $j: \got n' \to \got{so}(\got v)$ be defined as in~\eqref{eq:def_map_j}. If $L: \got n \to \got n$ is a skew-symmetric derivation, then
    \begin{equation*}
        [L, j(Z)] = j\bigl(L(Z)\bigr) \quad \text{for all } Z \in \got n'.
    \end{equation*}
\end{proposition}

\begin{proof}
    Let $X, Y \in \got v$. Then
    \begin{align*}
        \langle j(L(Z)) X, Y \rangle &= \langle [X,Y], L(Z) \rangle = -\langle L[X,Y], Z \rangle \\
                                     &= -\bigl(\langle [L(X), Y], Z \rangle + \langle [X, L(Y)], Z \rangle\bigr) \\
                                     &= -\langle j(Z) L(X), Y \rangle - \langle j(Z) X, L(Y) \rangle \\
                                     &= -\langle j(Z) L(X), Y \rangle + \langle L \, j(Z) X, Y \rangle = \langle [L, j(Z)](X), Y \rangle.
    \end{align*}
    Since $X, Y \in \got v$ are arbitrary, we conclude that $j(L(Z)) = [L, j(Z)]$.
\end{proof}

\section{Pseudo-Riemannian naturally reductive spaces}\label{sec:3_Pseudo-Riemannian_naturally_reductive_spaces}

Let $N$ be a simply connected (not necessarily $2$-step nilpotent) Lie group endowed with a left-invariant pseudo-Riemannian metric. Denote by $\Iso(N)$ the full isometry group of $N$ and let $H = \Iso(N)_e$ be the isotropy group at the identity element $e \in N$. It is standard that
\begin{equation*}
    \Iso(N) = L_N \cdot H,
\end{equation*}
where $L_N \simeq N$ is the subgroup of $\Iso(N)$ consisting of left-translations. Note that $H \cap L_N = \{ \Id \}$.

Consider the Lie subgroup $H^{\aut}$ of $H$ consisting of the isometric automorphisms of $N$, that is,
\begin{equation*}\label{eq:haut}
    H^{\aut} = \Aut(N) \cap \Iso(N) = \Aut(N) \cap H,
\end{equation*}
and the Lie subgroup $\Iso^{\aut}(N)$ of $\Iso(N)$ given by
\begin{equation*}
    \Iso^{\aut}(N) = L_N \cdot H^{\aut}.
\end{equation*}

\begin{proposition}\label{prop:Nnormal}
    Let $N$ be a simply connected Lie group endowed with a left-invariant pseudo-Riemannian metric. Then $\Iso^{\aut}(N) = N \rtimes H^{\aut}$ is a semidirect product, and it is the normalizer of $N$ in $\Iso(N)$.
\end{proposition}

\begin{proof}
    First, consider the subgroup $\Iso^{\aut}(N) = L_N \cdot H^{\aut}$. Since $H^{\aut} \subset \Iso(N)_e$, we have $N \cap H^{\aut} = \{e\}$. Moreover, for any $g, n \in N$ and $h \in H^{\aut}$,
    \begin{equation*}
        (L_g h) L_n (L_g h)^{-1} = L_g (h L_n h^{-1}) L_{g^{-1}} = L_g L_{h(n)} L_{g^{-1}} = L_{g \, h(n) \, g^{-1}},
    \end{equation*}
    which belongs to $N$. This shows that $N$ is a normal subgroup of $\Iso^{\aut}(N)$, and therefore $\Iso^{\aut}(N) = N \rtimes H^{\aut}$.

    Let $G$ be a subgroup of $\Iso(N)$ such that $N$ is a normal subgroup of $G$. Then $G = N \rtimes K$, where $K$ is a subgroup of the isotropy subgroup $\Iso(N)_e$.

    We claim that $K \subset H^{\aut}$. Let $h \in K$ and $g \in N$. Since $h$ normalizes $N$, there exists $g' \in N$ such that
    \begin{equation*}
        h \circ L_g \circ h^{-1} = L_{g'}.
    \end{equation*}
    Evaluating both sides at the identity $e$, we obtain
    \begin{equation*}
        g' = L_{g'}(e) = h \circ L_g \circ h^{-1}(e) = h\bigl(g \cdot h^{-1}(e)\bigr) = h(g),
    \end{equation*}
    and therefore
    \begin{equation*}
        h \circ L_g \circ h^{-1} = L_{h(g)}, \qquad \text{equivalently,} \qquad h \circ L_g = L_{h(g)} \circ h.
    \end{equation*}
    Thus, for any $a \in N$,
    \begin{equation*}
        h(g a) = h(L_g(a)) = L_{h(g)}(h(a)) = h(g) h(a),
    \end{equation*}
    which shows that $h$ is an automorphism of $N$. Hence $K \subset H^{\aut}$, and so $G \subseteq \Iso^{\aut}(N)$. Since $N$ is a normal subgroup of $\Iso^{\aut}(N)$, we conclude that $\Iso^{\aut}(N)$ is the normalizer of $N$ in $\Iso(N)$.
\end{proof}

Let now $N$ be a simply connected $2$-step nilpotent Lie group endowed with a left-invariant metric. If $N$ with this metric is a Riemannian manifold, then $\Iso(N) = \Iso^{\aut}(N)$ (cf.~\cite{Wolf1962}). This is no longer true for a pseudo-Riemannian nilmanifold (cf.~\cite{DelBarcoOvando2014}); however, by Proposition~\ref{prop:Nnormal}, $\Iso^{\aut}(N)$ is the normalizer of $N$ in $\Iso(N)$ and acts transitively on $N$ (since it contains all left translations).

Let $\got n$ be the Lie algebra of $N$, with the metric induced by that of $N$. Since $N$ is simply connected, $\Aut(N) \simeq \Aut(\got n)$. Therefore
\begin{equation*}
    H^{\aut} \simeq \OO(\got n) \cap \Aut(\got n),
\end{equation*}
where $\OO(\got n)$ is the orthogonal group of $\got n$ with respect to the given metric. With these identifications, the Lie algebra of $\Iso^{\aut}(N)$ is $\got{iso}^{\aut}(N) \simeq \got n \rtimes \got h^{\aut}$, where
\begin{equation}\label{eq:haut_lie}
    \got h^{\aut} = \operatorname{Der}(\got n) \cap \got{so}(\got n)
\end{equation}
is the Lie algebra of skew-symmetric derivations of $\got n$.

Recall that, under the identification $\got{iso}^{\aut}(N) \simeq \got n \rtimes \got h^{\aut}$, if $U, V \in \got n$ and $A, B \in \got h^{\aut}$, the Lie bracket of $\got{iso}^{\aut}(N)$ is given by
\begin{align}\label{eq:corcheteiso}
    [U, V]_{\got{iso}^{\aut}(N)} = [U, V]_{\got n}, && [A, B]_{\got{iso}^{\aut}(N)} = [A, B]_{\got h^{\aut}}, && [A, U]_{\got{iso}^{\aut}(N)} = A(U).
\end{align}

Since $\Iso^{\aut}(N)$ acts transitively on $N$ and $\Iso^{\aut}(N)_e = H^{\aut}$,
\begin{equation*}
    N = \Iso^{\aut}(N) / H^{\aut}
\end{equation*}
is a pseudo-Riemannian homogeneous space. We are interested in determining conditions under which $N$ is naturally reductive with respect to the presentation group $\Iso^{\aut}(N)$.

Let $M = G/K$ be a pseudo-Riemannian homogeneous space, where $G$ is a Lie subgroup of $\Iso(M)$ and $K = G_e$ is the isotropy subgroup at the identity $e$. If there exists an $\ad(\got k)$-invariant subspace $\got m \subset \got g$ such that
\begin{equation}\label{eq:desc_reduct}
    \got g = \got k \oplus \got m,
\end{equation}
$\got m$ is called a \emph{reductive complement}, and the decomposition~\eqref{eq:desc_reduct} a \emph{reductive decomposition} of $\got g$.

Given a reductive complement $\got m$, the pseudo-Riemannian metric on $M$ induces an $\Ad(K)$-invariant non-degenerate symmetric bilinear form on $\got m$, which we denote by $\langle \cdot, \cdot \rangle_{\got m}$. More precisely, identifying $\got m$ with the tangent space $T_e M$ via the natural projection $\pi: G \to G/K$, the restriction of the metric at $e$ defines an inner product on $\got m$ that is invariant under the isotropy representation.

Recall that $M$ is \emph{naturally reductive} with respect to $G$ if there exists a reductive decomposition as in~\eqref{eq:desc_reduct} such that, for all $U, V, W \in \got m$,
\begin{equation}\label{eq:DefNatRed}
    \langle [U,V]_{\got m}, W \rangle + \langle V, [U,W]_{\got m} \rangle = 0,
\end{equation}
where $[\cdot,\cdot]_{\got m}$ denotes the $\got m$-component of the Lie bracket in $\got g$.

To determine whether a homogeneous space is naturally reductive without explicitly producing a reductive complement, we now introduce the \emph{Geodesic Lemma} (cf. \cite[Proposition 1]{NikolayevskyWolf2022}), which describes naturally reductive spaces as a particular class of G.O.\ spaces.

\begin{lemma}[Geodesic Lemma]\label{lema:geodesicLemma}
    Let $M = G/K$ be a pseudo-Riemannian homogeneous space, where $G$ is a Lie subgroup of $\Iso(M)$ and $K = G_e$ is the isotropy subgroup at the identity $e$. Suppose that $M$ admits a reductive complement $\got m$. Then $M$ is a naturally reductive space (possibly with respect to a different reductive decomposition) if and only if there exists a linear and $\ad(\got k)$-equivariant map $\xi: \got m \to \got k$ such that
    \begin{equation}\label{eq:geodesic_graph}
        \langle [X + \xi(X), Z]_{\got m}, X \rangle_{\got m} = 0 \quad \text{for all } X, Z \in \got m,
    \end{equation}
    where $[\cdot,\cdot]_{\got m}$ denotes the $\got m$-component of the Lie bracket in $\got g$.
\end{lemma}

In Lemma~\ref{lema:geodesicLemma}, $\ad(\got k)$-equivariant means that, for each $L \in \got k$ and each $X \in \got m$, $\xi([L,X]) = [L, \xi(X)]$, where $[\cdot,\cdot]$ is the Lie bracket in $\got g$.

The map $\xi: \got m \to \got k$ in Lemma~\ref{lema:geodesicLemma} is called a \emph{geodesic graph}. If $M$ is a Riemannian homogeneous space, one only requires the existence of a linear map $\xi: \got m \to \got k$ such that~\eqref{eq:geodesic_graph} holds (see~\cite{ChenNikolayevskyNikonorov2022}).

We will use Lemma~\ref{lema:geodesicLemma} to give conditions under which $N$ is naturally reductive with respect $\Iso^{\aut}$. This generalizes the conditions found in~\cite{Ovando2013}, even when $\got z$ is degenerate.

\begin{lemma}\label{lemma:nat_red_cond}
    Let $N$ be a $2$-step nilpotent Lie group with Lie algebra $\got n$, endowed with a left-invariant pseudo-Riemannian metric such that the commutator $\got n'$ is non-degenerate. Let $\got v$ be an orthogonal complement of $\got n'$ in $\got n$, and let $j: \got n' \to \got{so}(\got v)$ be defined by~\eqref{eq:def_map_j}. Let $H = H^{\aut}$ and $\got h=\got h^{\aut}$. Then the following statements are equivalent:
    \begin{enumerate}
        \item\label{lemma:nat_red_cond_item_1} $N$ is naturally reductive with respect to $G = N \rtimes H$.
        \item\label{lemma:nat_red_cond_item_2} There exists a linear and $\ad(\got h)$-equivariant map $\tilde{\xi}: \got n' \to \got h$ such that
        \begin{equation*}
            j(Z) = \tilde{\xi}(Z)\big|_{\got v}, \quad \text{for all } Z \in \got n'.
        \end{equation*}
        \item\label{lemma:nat_red_cond_item_3} There exists a linear map $\tau: \got n' \to \got{so}(\got n')$ such that
        \begin{equation*}
            [j(Z), j(Z')] = j\bigl(\tau_Z(Z')\bigr), \quad \text{for all } Z, Z' \in \got n'.
        \end{equation*}
    \end{enumerate}
\end{lemma}

\begin{proof}
    \noindent $\eqref{lemma:nat_red_cond_item_1} \Rightarrow \eqref{lemma:nat_red_cond_item_2}$. Suppose that $N$ is naturally reductive. By Lemma~\ref{lema:geodesicLemma}, there exists a linear, $\ad(\got h)$-equivariant map $\xi: \got n \to \got h$ such that, for every $A, B \in \got n$,
    \begin{equation}\label{eq:proof_lema_con_geo_lemma}
        0 = \langle [A + \xi(A), B], A \rangle = \langle [A,B], A \rangle + \langle \xi(A) B, A \rangle.
    \end{equation}
    Since $\xi$ is $\ad(\got h)$-equivariant, for every $L \in \got h$ and $A \in \got n$,
    \begin{equation*}
        [L, \xi(A)]_{\got{iso}(N)} = \xi([L,A]_{\got{iso}(N)}) = \xi(L(A)).
    \end{equation*}
    Let $X, Y \in \got v$. Then $[X,Y] \in \got n'$, and so $\langle [X,Y], X \rangle = 0$. Then from~\eqref{eq:proof_lema_con_geo_lemma}, taking $A = X$ and $B = Y$,
    \begin{equation*}
        \langle \xi(X) Y, X \rangle = 0.
    \end{equation*}

    Let $Z \in \got n'$. Taking $A = X + Z$ and $B = Y$ in~\eqref{eq:proof_lema_con_geo_lemma}, since $\xi(A)$ preserves the decomposition $\got n = \got n' \oplus \got v$, we obtain
    \begin{equation*}
        \begin{array}{ll}
            0 &= \langle [X + Z, Y], X + Z \rangle + \langle \xi(X+Z) Y, X + Z \rangle \\
              &= \langle [X,Y], Z \rangle + \langle \xi(X) Y, X \rangle + \langle \xi(Z) Y, X \rangle \\
              &= \langle (j(Z) - \xi(Z)) X, Y \rangle.
        \end{array}
    \end{equation*}
    Since $X, Y \in \got v$ are arbitrary,
    \begin{equation*}
        \xi(Z)\big|_{\got v} = j(Z), \quad \text{for every } Z \in \got n'.
    \end{equation*}
    It suffices to set $\tilde{\xi} = \xi|_{\got n'}$.

    \medskip

    \noindent $\eqref{lemma:nat_red_cond_item_2} \Rightarrow \eqref{lemma:nat_red_cond_item_3}$. Let $Z, Z' \in \got n'$. Since $\tilde{\xi}$ is $\ad(\got h)$-equivariant,
    \begin{equation*}
        [j(Z), j(Z')] = [\tilde{\xi}(Z), \tilde{\xi}(Z')]\big|_{\got v} = \tilde{\xi}(\tilde{\xi}(Z) Z')\big|_{\got v} = j(\tilde{\xi}(Z) Z').
    \end{equation*}
    It suffices to set $\tau_Z = \tilde{\xi}(Z)|_{\got n'}$.

    \noindent $\eqref{lemma:nat_red_cond_item_3} \Rightarrow \eqref{lemma:nat_red_cond_item_1}$. Let $Z \in \got n'$. Define a linear map $\tilde{\xi}(Z): \got n \to \got n$ by
    \begin{equation*}
        \tilde{\xi}(Z) A = \begin{cases}
            \tau_Z(A) & \text{if } A \in \got n', \\
            j(Z) A & \text{if } A \in \got v.
        \end{cases}
    \end{equation*}

    Let us show that $\tilde{\xi}(Z) \in \got h$. Let $A, B \in \got n$ be arbitrary. If $A \in \got n'$, then $[A,B] = 0$, $[\tilde{\xi}(Z) A, B] = 0$ and $[A, \tilde{\xi}(Z) B] = 0$. Therefore,
    \begin{equation*}
        0 = \tilde{\xi}(Z) [A,B] = [\tilde{\xi}(Z) A, B] + [A, \tilde{\xi}(Z) B].
    \end{equation*}
    If $A \in \got v$ and $B \in \got n'$, a similar argument applies. For the case $A, B \in \got v$, if $Z' \in \got n'$, then
    \begin{equation*}
        \begin{array}{ll}
            \langle \tilde{\xi}(Z) [A,B], Z' \rangle &= \langle \tau_Z([A,B]), Z' \rangle = -\langle [A,B], \tau_Z(Z') \rangle \\
            &= -\langle j(\tau_Z(Z')) A, B \rangle = -\langle [j(Z), j(Z')] A, B \rangle,
        \end{array}
    \end{equation*}
    and
    \begin{equation*}
        \begin{array}{ll}
            \langle [\tilde{\xi}(Z) A, B] + [A, \tilde{\xi}(Z) B], Z' \rangle &= \langle [j(Z) A, B], Z' \rangle - \langle [j(Z) B, A], Z' \rangle \\
            &= \langle j(Z') j(Z) A, B \rangle - \langle j(Z') j(Z) B, A \rangle \\
            &= \langle j(Z') j(Z) A, B \rangle - \langle B, j(Z) j(Z') A \rangle \\
            &= \langle (j(Z') j(Z) - j(Z) j(Z')) A, B \rangle.
        \end{array}
    \end{equation*}
    Therefore,
    \begin{equation*}
        \langle \tilde{\xi}(Z) [A,B], Z' \rangle = \langle [\tilde{\xi}(Z) A, B] + [A, \tilde{\xi}(Z) B], Z' \rangle
    \end{equation*}
    for every $Z' \in \got n'$. We conclude that $\tilde{\xi}(Z)$ is a derivation of $\got n$ for every $Z \in \got n'$.

    Since $\tilde{\xi}(Z)$ preserves the orthogonal decomposition $\got n = \got n' \oplus \got v$ and is skew-symmetric on each summand, $\tilde{\xi}(Z) \in \got{so}(\got n)$. Therefore $\tilde{\xi}(Z) \in \got h$. Furthermore, $Z \mapsto \tilde{\xi}(Z)$ is linear, since $Z \mapsto \tau_Z$ and $Z \mapsto j(Z)$ are linear.

    Define a linear map $\xi: \got n \to \got h$ by
    \begin{equation*}
        \xi(A) = \left\{
        \begin{array}{cl}
            \tilde{\xi}(Z), & \text{if } A = Z \in \got n', \\[4pt]
            0,              & \text{if } A \in \got v.
        \end{array}
        \right.
    \end{equation*}
    Let us show that $\xi$ is $\ad(\got h)$-equivariant; that is, for every $L \in \got h$ and every $A \in \got n$,
    \begin{equation*}
        [L, \xi(A)] = \xi(L(A)).
    \end{equation*}
    First, observe that if $L \in \got h$, then $L$ preserves the orthogonal decomposition $\got n = \got n' \oplus \got v$. Hence, if $A = X \in \got v$, then $\xi(L(X)) = 0$ and $[L, \xi(X)] = 0$, so the equality holds trivially.

    Assume that $A = Z \in \got n'$. Let $L \in \got h$. If $X \in \got v$, by Proposition~\ref{prop:lightl_vect_skew-sym_der},
    \begin{equation*}
        [L, \xi(Z)](X) = [L, j(Z)](X) = j(L(Z)) X = \xi(L(Z))(X).
    \end{equation*}
    Hence,
    \begin{equation*}
        [L, \xi(Z)]\big|_{\got v} = \xi(L(Z))\big|_{\got v}.
    \end{equation*}

    On the other hand, if $Z' \in \got n'$, then
    \begin{equation*}
        [L, \xi(Z)](Z') = [L, \tau_Z](Z') = L(\tau_Z(Z')) - \tau_Z(L(Z')).
    \end{equation*}
    By Proposition~\ref{prop:lightl_vect_skew-sym_der} and item~\eqref{lemma:nat_red_cond_item_3},
    \begin{align*}
        j(L(\tau_Z(Z'))) &= [L, j(\tau_Z(Z'))] = [L, [j(Z), j(Z')]] \\
                         &= [[L, j(Z)], j(Z')] + [j(Z), [L, j(Z')]] \\
                         &= [j(L(Z)), j(Z')] + [j(Z), j(L(Z'))] \\
                         &= j(\tau_{L(Z)}(Z')) + j(\tau_Z(L(Z'))).
    \end{align*}
    Therefore,
    \begin{equation*}
        j([L, \xi(Z)](Z')) = j(L(\tau_Z(Z'))) - j(\tau_Z(L(Z'))) = j(\tau_{L(Z)}(Z')).
    \end{equation*}
    Since $j$ is injective,
    \begin{equation*}
        [L, \xi(Z)](Z') = \tau_{L(Z)}(Z') = \xi(L(Z))(Z').
    \end{equation*}
    We conclude that $\xi$ is $\ad(\got h)$-equivariant.

    Finally, we show that $\xi:\got n\to\got h$ is a geodesic graph for $\got n$, i.e. $\langle [A + \xi(A), B]_{\got n}, A \rangle = 0$ for all $A,B\in \got n$. Let
    \begin{equation*}
        A=X+Z,\qquad X\in\got v,\quad Z\in\got n'.
    \end{equation*}
    Since $\xi(X)=0$, by linearity we have $\xi(A)=\xi(Z)$.
    
    It is enough to consider separately the cases $B=Y\in\mathfrak v$ and $B=Z'\in\got n'$.
    
    Assume first that $B=Y\in\got v$. Since $\got n$ is $2$-step nilpotent, $[Z,Y]=0$, and therefore
    \begin{equation*}
        [A+\xi(A),Y] = [X,Y]+\xi(Z)Y = [X,Y]+j(Z)Y.
    \end{equation*}
    Here $[X,Y]\in\got n'$ and $j(Z)Y\in\got v$. Hence, using the orthogonal decomposition $\got n=\got n'\oplus\got v$ and since $j(Z)\in\got{so}(\got v)$, we obtain
    \begin{align*}
        \left\langle [A+\xi(A),Y],A\right\rangle &= \left\langle [X,Y],Z\right\rangle + \left\langle j(Z)Y,X\right\rangle\\ &= \left\langle j(Z)X,Y\right\rangle + \left\langle j(Z)Y,X\right\rangle=0.
    \end{align*}
    
    Now let $B=Z'\in\got n'$. Since $\got n'$ is central,
    \begin{equation*}
        [A+\xi(A),Z'] = \xi(Z)Z' = \tau_Z(Z').
    \end{equation*}
    Therefore,
    \begin{align*}
        \left\langle [A+\xi(A),Z'],A\right\rangle &= \left\langle \tau_Z(Z'),Z\right\rangle= -\left\langle Z',\tau_Z(Z)\right\rangle.
    \end{align*}
    But $j(\tau_Z(Z))=[j(Z),j(Z)]=0$, and, since $j$ is injective, $\tau_Z(Z)=0$. Thus
    \begin{equation*}
        \left\langle [A+\xi(A),Z'],A\right\rangle=0.
    \end{equation*}
    
    By linearity in $B$, it follows that $\left\langle [A+\xi(A),B],A\right\rangle=0$ for every $A,B\in\got n$. Hence $\xi$ is an $\operatorname{ad}(\got h)$-equivariant geodesic graph, and therefore $N$ is naturally reductive.
\end{proof}

Let $N$ be a simply connected $2$-step nilpotent Lie group with a left-invariant metric that is naturally reductive with respect to $\Iso^{\aut}(N)$, such that $\got n'$ is non-degenerate. For each $Z \in \got n'$, let $\tau_Z \in \got{so}(\got n')$ be the map given by item~\eqref{lemma:nat_red_cond_item_3} of Lemma~\ref{lemma:nat_red_cond}. Since $j$ is injective on $\got n'$, item~\eqref{lemma:nat_red_cond_item_3} of Lemma~\ref{lemma:nat_red_cond} easily implies that  
\begin{equation}\label{eq:corchete_en_n'}
    [Z, Z']_0 = \tau_Z(Z') \quad \text{for all } Z, Z' \in \got n'.
\end{equation}
defines a Lie bracket $[\cdot,\cdot]_0$ on $\got n'$ (cf. \cite{Ovando2013}). Furthermore, if $A, B, C \in \got n'$, then
\begin{equation*}
    \langle [A,B]_0, C \rangle = \langle \tau_A(B), C \rangle = -\langle B, \tau_A(C) \rangle = -\langle B, [A,C]_0 \rangle.
\end{equation*}
Thus $\langle \cdot, \cdot \rangle|_{\got n'}$ is an $\ad$-invariant metric on $(\got n', [\cdot,\cdot]_0)$.

From Proposition~\ref{prop:not_split-off_a_flat_factor} we obtain:

\begin{theorem}\label{teo:summarize_nat_red}
    Let $N$ be a simply connected $2$-step nilpotent Lie group with Lie algebra $\got n$, endowed with a left-invariant Lorentzian metric such that $\got n'$ is non-degenerate. Assume that $N$ is naturally reductive with respect to $\Iso^{\aut}(N)$ and that $N$ does not split off a flat factor. Let $j: \got n' \to \got{so}(\got v)$ be the linear map defined in~\eqref{eq:def_map_j}, and let $[\cdot,\cdot]_0$ be the Lie bracket on $\got n'$ defined in~\eqref{eq:corchete_en_n'}. Then:
    \begin{enumerate}
        \item $\got n'$ admits a Lie bracket $[\cdot,\cdot]_0$ such that $(\got n', [\cdot,\cdot]_0)$ is a Lie algebra with an $\ad$-invariant metric.
        \item $j: (\got n', [\cdot,\cdot]_0) \to \got{so}(\got v)$ is a faithful representation such that:
        \begin{enumerate}
            \item if $\got z$ is non-degenerate, then $\got z = \got n'$ and $j$ has no trivial subrepresentations, i.e.\ $\bigcap_{Z \in \got n'} \ker j(Z)$ is trivial;
            \item if $\got z$ is degenerate, then $\got z = \got n' \oplus \R E$, where $E$ is a lightlike vector, and $\bigcap_{Z \in \got n'} \ker j(Z) = \R E$ is a one-dimensional degenerate subspace of $\got v$.
        \end{enumerate}
    \end{enumerate}
\end{theorem}

\section{Naturally reductive Lorentzian $2$-step nilpotent Lie groups via representations}\label{sec:4_Naturally_reductive_Lorentzian}

In this section we generalize the construction of a $2$-step nilpotent Lie algebra $\got n$ from a representation $\pi: \got g \to \End(\got v)$, where $\got g$ is a Lie algebra with an $\ad$-invariant metric and $\got v$ is a metric vector space, in such a way that the associated $2$-step nilpotent simply connected Lie group $N$ is naturally reductive with respect to the presentation group $\Iso^{\aut}(N)$ and does not split off a flat factor (cf.~\cite{Lauret1999, Ovando2013}). Since, by Theorem~\ref{teo:summarize_nat_red}, the structure of the Lie algebra $\got n$ depends on whether $\got z$ is degenerate or non-degenerate, we introduce two types of constructions to cover both cases.

The following definition was introduced in~\cite{Lauret1999} for the Riemannian case and generalized by Ovando in the pseudo-Riemannian setting in~\cite{Ovando2013}.

\begin{definition}
    A \emph{data set} is a triple $(\got g, \got v, \pi)$ such that:
    \begin{enumerate}
        \item $\got g$ is a Lie algebra endowed with an $\ad$-invariant metric $\langle \cdot, \cdot \rangle_{\got g}$ (i.e., $\ad_Z \in \got{so}(\got g, \langle \cdot, \cdot \rangle_{\got g})$ for each $Z \in \got g$);
        \item $\got v$ is a real vector space;
        \item $\pi: \got g \to \End(\got v)$ is a real faithful representation;
        \item $\got v$ is endowed with a $\pi(\got g)$-invariant metric $\langle \cdot, \cdot \rangle_{\got v}$ (i.e., $\pi: \got g \to \got{so}(\got v)$).
    \end{enumerate}
\end{definition}

Given a data set $(\got g, \got v, \pi)$, define
\begin{equation*}
    \got n = \got g \oplus \got v,
\end{equation*}
and equip $\got n$ with the metric
\begin{equation}\label{eq:metricNR}
    \langle \cdot, \cdot \rangle\big|_{\got g} = \langle \cdot, \cdot \rangle_{\got g}, \quad \langle \cdot, \cdot \rangle\big|_{\got v} = \langle \cdot, \cdot \rangle_{\got v}, \quad \langle \got g, \got v \rangle = 0.
\end{equation}
One can define a Lie bracket on $\got n$ by
\begin{equation}\label{eq:bracketNR}
    \left\{
    \begin{array}{l}
        [\got g, \got n] = \{0\}, \quad [\got v, \got v] \subset \got g, \\[2pt]
        \langle [X,Y], Z \rangle = \langle \pi(Z) X, Y \rangle \quad \text{for all } Z \in \got g, \, X, Y \in \got v.
    \end{array}
    \right.
\end{equation}
Then $\got n$ is a $2$-step nilpotent metric Lie algebra, which we denote by $\got n(\got g, \got v, \pi)$. We will denote by $N(\got g, \got v, \pi)$ the simply connected $2$-step nilpotent Lie group whose Lie algebra is $\got n(\got g, \got v, \pi)$.

\begin{definition}\label{defi:data_set_degenerate}
    Let $(\got g, \got v, \pi)$ be a data set, and let $\got n = \got g \oplus \got v$ be the Lie algebra defined by~\eqref{eq:bracketNR} with the metric $\langle \cdot, \cdot \rangle$ defined by~\eqref{eq:metricNR}. Assume that $\langle \cdot, \cdot \rangle$ is Lorentzian (i.e., it has signature $1$). Then:
    \begin{itemize}
        \item if $\pi: \got g \to \got{so}(\got v)$ has no trivial subrepresentations (i.e., $\bigcap_{Z \in \got g} \ker \pi(Z) = \{0\}$), we say that $(\got g, \got v, \pi)$ is a \emph{Lorentzian data set};
        \item if $\bigcap_{Z \in \got g} \ker \pi(Z) = \R E$ for some $E \in \got v$ with $E \neq 0$ and $\langle E, E \rangle = 0$, then $(\got g, \got v, \pi)$ is called a \emph{degenerate Lorentzian data set}.
    \end{itemize}
\end{definition}

\begin{observation}
    The definition of a Lorentzian data set coincides with the definition of a data set given in~\cite{Ovando2013} when the metric on $\got n$ is Lorentzian.
\end{observation}

\begin{proposition}\label{prop:centro_no_deg_dataset}
    Let $(\got g, \got v, \pi)$ be a Lorentzian data set, and let $\got n = \got n(\got g, \got v, \pi)$ be the $2$-step nilpotent metric Lie algebra associated with $(\got g, \got v, \pi)$. Then:
    \begin{enumerate}
        \item\label{item1propdataset} the center $\got z$ of $\got n$ coincides with $\got g$ and is therefore non-degenerate;
        \item\label{item2propdataset} the commutator $\got n'$ is non-degenerate, and $\got n' = \got z$.
    \end{enumerate}
\end{proposition}

\begin{proof}
    It is immediate that the center $\got z$ of $\got n$ contains $\got g$. From~\eqref{eq:metricNR} and~\eqref{eq:bracketNR}, if $X \in \got v$ belongs to $\got z$, then $X \in \bigcap_{Z \in \got g} \ker \pi(Z) = \{0\}$, and hence $X = 0$. Therefore $\got z = \got g$, and the center of $\got n(\got g, \got v, \pi)$ is non-degenerate. This proves~\eqref{item1propdataset}.

    Suppose now that there exists $Z \in \got n'$ such that $\langle Z', Z \rangle = 0$ for all $Z' \in \got n'$. In particular, $\langle [X,Y], Z \rangle = 0$ for all $X, Y \in \got v$. By~\eqref{eq:bracketNR}, $\langle \pi(Z) X, Y \rangle = 0$ for all $X, Y \in \got v$, and since the metric on $\got v$ is non-degenerate, $\pi(Z) X = 0$ for all $X \in \got v$. Hence $\pi(Z) = 0$, which is impossible because $\pi$ is faithful. So $\got n'$ is non-degenerate. By the same argument, $\got n' = \got z$; otherwise there would exist $Z \in \got z$ with $Z \perp \got n'$, and we would get $\pi(Z) = 0$.
\end{proof}

\begin{proposition}\label{prop:centro_deg_dataset}
    Let $(\got g, \got v, \pi)$ be a degenerate Lorentzian data set, and let $\got n = \got n(\got g, \got v, \pi)$ be the $2$-step nilpotent metric Lie algebra associated with $(\got g, \got v, \pi)$. Then:
    \begin{enumerate}
        \item\label{item1propdataset_deg} the metric $\langle \cdot, \cdot \rangle_{\got g}$ on $\got g$ is Riemannian, and the metric $\langle \cdot, \cdot \rangle_{\got v}$ on $\got v$ is Lorentzian;
        \item\label{item2propdataset_deg} the commutator $\got n'$ is non-degenerate, and $\got n' = \got g$;
        \item\label{item3propdataset_deg} the center $\got z$ is degenerate, and $\got z = \got n' \oplus \R E$.
    \end{enumerate}
\end{proposition}

\begin{proof}
    Item~\eqref{item1propdataset_deg} follows immediately from Definition~\ref{defi:data_set_degenerate}.

    From~\eqref{eq:bracketNR}, the commutator subalgebra $\got n'$ is contained in $\got g$. Suppose $\got n' \subsetneq \got g$. Then there exists $Z \in \got g$ orthogonal to $\got n'$. Thus, for all $X, Y \in \got v$,
    \begin{equation*}
        0 = \langle [X,Y], Z \rangle = \langle \pi(Z) X, Y \rangle.
    \end{equation*}
    Hence $\pi(Z) = 0$ and so $Z = 0$. Therefore $\got g = \got n'$, proving~\eqref{item2propdataset_deg}.

    Since $E$ is central, $\got n' \oplus \R E \subset \got z$. Let $A \in \got z$, and write $A = X + Z$ with $X \in \got v$ and $Z \in \got g = \got n'$. By~\eqref{eq:bracketNR}, $Z \in \got z$ and hence $X \in \got z$ as well. Thus, for every $Y \in \got v$ and $Z' \in \got g$,
    \begin{equation*}
        0 = \langle [X,Y], Z' \rangle = \langle \pi(Z') X, Y \rangle.
    \end{equation*}
    Therefore $X \in \bigcap_{Z' \in \got g} \ker \pi(Z') = \R E$, and consequently $A \in \got n' \oplus \R E$.
\end{proof}

Let $N$ be a simply connected $2$-step nilpotent Lie group endowed with a left-invariant Lorentzian metric such that $\got n'$ is non-degenerate, and assume that $N$ is naturally reductive with respect to $\Iso^{\aut}(N)$. 
From Theorem~\ref{teo:summarize_nat_red}, if $N$ does not split off a flat factor, then $N$ gives rise to a Lorentzian or a degenerate Lorentzian data set, with $\got g = \got n'$ (with an appropriate Lie bracket) and $\pi = j$.

We now show that the converse also holds. This result was established by Ovando~\cite{Ovando2013} in the Lorentzian setting under the hypothesis that the center $\got z$ is non-degenerate and $j$ is injective. We will prove that the statement holds even when $\got z$ is degenerate.

\begin{proposition}\label{prop:geodesic_graph_dataset}
    Let $(\got g, \got v, \pi)$ be a Lorentzian or a degenerate Lorentzian data set, and let $\got n = \got n(\got g, \got v, \pi)$. Then $N = N(\got g, \got v, \pi)$, endowed with the left-invariant metric induced by~\eqref{eq:metricNR}, is naturally reductive with respect to the Lie group $\Iso^{\aut}(N)$, and $N$ does not split off a flat factor. 
\end{proposition}

\begin{proof}
     Since $\got n' = \got g$, the definition of $\pi$ immediately gives that $j: \got n' \to \got{so}(\got v)$ coincides with $\pi$. Setting $\tau_Z(Z') = [Z, Z']_{\got g}$, and using that $\pi$ is a representation, we obtain
    \begin{equation*}
        [j(Z), j(Z')] = [\pi(Z), \pi(Z')] = \pi([Z,Z']_{\got g}) = j(\tau_Z(Z')).
    \end{equation*}
    By Lemma~\ref{lemma:nat_red_cond}, $N$ is naturally reductive. Hence, it is geodesically complete. So Propositions \ref{prop:centro_no_deg_dataset}, \ref{prop:centro_deg_dataset} and \ref{prop:not_split-off_a_flat_factor} imply that  $N$ does not split off a flat factor.
\end{proof}

\begin{observation} 
    The implications $(1)\Rightarrow (2)\Rightarrow (3)$ in the proof of Lemma \ref{lemma:nat_red_cond} are valid for any $H\subseteq H^{\aut}$. The implication $(3) \Rightarrow (1)$  produces a geodesic graph $\xi$ taking values already in the image of $\got g \to \got h^{\aut}$, $Z \mapsto(\ad_Z,\pi(Z))$. Hence,  $N$ is naturally reductive with respect to  $\Iso^{\aut}(N)$ if and only if it is naturally reductive with respect to  $N\rtimes G_{\pi}$, where $G_{\pi}$ is the connected subgroup of $H^{\aut}$ with Lie algebra
    \begin{equation*}
        \{ (\ad_Z,\pi(Z)):Z\in\got g \}\ \cong \got g \subset \got h^{\aut}.
    \end{equation*}
\end{observation}

\begin{proposition}\label{prop:gcompact}
    Let $(\got g, \got v, \pi)$ be a Lorentzian or a degenerate Lorentzian data set. Then $\got g$ is a compact Lie algebra. Hence $\got g = \got g' \oplus \got c$, where $\got c$ is the center of $\got g$ and $\got g' = [\got g, \got g]$ is semisimple.
\end{proposition}

\begin{proof}
    If $(\got g, \got v, \pi)$ is a degenerate Lorentzian data set, the Lie algebra $\got g$ is Riemannian. If $(\got g, \got v, \pi)$ is a Lorentzian data set, $\got g$ is non-degenerate, and hence either Riemannian or Lorentzian. If $\got g$ is Riemannian and $\got v$ is Lorentzian, then $\got g$ is compact, since the metric on $\got g$ is Riemannian and $\ad$-invariant. If $\got g$ is Lorentzian, then since the representation $\pi$ is faithful, $\got g$ is isomorphic to $\pi(\got g) \subset \got{so}(\got v)$. Since $\got v$ is Riemannian, $\got{so}(\got v)$ is compact and  $\got g$ is isomorphic to a subalgebra of a compact Lie algebra and hence is compact.
\end{proof}

\section{Degenerate data sets}\label{sec:5_Degenerate_data_set}

In this section we show that any Lie algebra $\got n = \got n(\got g, \got v, \pi)$, where $(\got g, \got v, \pi)$ is a degenerate Lorentzian data set, is a central extension of a semidirect product whose underlying factor is the Lie algebra $\got m_0$ of a Riemannian nilpotent naturally reductive Lie group.

\begin{theorem}\label{teo:central_extension_semidirect}
    Let $N$ be a simply connected nilpotent Lie group endowed with a left-invariant Lorentzian metric such that $N$ is naturally reductive with respect to $\Iso^{\aut}(N)$, and suppose that $N$ does not split off a flat factor. Let $\got n$ be the Lie algebra of $N$, and assume that $\got n'$ is non-degenerate, the center $\got z$ of $\got n$ is degenerate, and $\got z = \got n' \oplus \R E$ with $\langle E, E \rangle = 0$. Let $F \in (\got n')^{\perp}$ be a lightlike vector such that $\langle E, F \rangle = 1$. Then
    \begin{equation*}
        \got n \simeq (\got m_0 \rtimes \R F) \oplus \R E,
    \end{equation*}
    where $\got m_0$ is the Lie algebra of a nilpotent Lie group endowed with a left-invariant Riemannian metric that is naturally reductive.
\end{theorem}

\begin{proof} Since $\got n'$ is non degenerate, $\got n$ is $2$-step nilpotent from \cite{NikolayevskyWolf2022}.
    Since $\got z$ is degenerate, we can write $\got z = \got n' \oplus \R E$, where $E$ is a lightlike vector spanning the radical of $\got z$; in particular, $E$ is orthogonal to $\got z$. The subspace $\got n'$ is non-degenerate, hence its orthogonal complement $(\got n')^{\perp}$ is also non-degenerate, and the restricted metric $\langle \cdot, \cdot \rangle|_{(\got n')^{\perp}}$ is Lorentzian. In particular, $E \in (\got n')^{\perp}$, and there exists $F \in (\got n')^{\perp}$ such that $\langle E, F \rangle = 1$. The subspace $\got n'$ is contained in $(\operatorname{span}\{E,F\})^{\perp}$. Hence there exists a subspace $\got b$ such that $\got n$ decomposes as the orthogonal direct sum
    \begin{equation*}
        \got n = \got n' \oplus \operatorname{span}\{E,F\} \oplus \got b.
    \end{equation*}

    Define $\got m_0 = \got n' \oplus \got b$ and $\rho = \ad_F|_{\got m_0}$. Since $[\got b, \got b] \subset \got n'$, $\got m_0$ is a Lie subalgebra of $\got n$. The derivation $\rho$ satisfies  $\rho(\got n') = \{0\}$ and $\rho(\got b) \subset \got n'$. Hence $\rho$ is a $2$-step nilpotent derivation of $\got m_0$.

    Therefore $\got n$ is  isomorphic  as a Lie algebra to the central extension $(\got m_0\ltimes_{\rho} \mathbb F)\oplus \mathbb R E$ of  $\got m_0\ltimes_{\rho} \R F$. It remains to prove that $\got m_0$ is naturally reductive. Observe that the restricted metric $\langle \cdot, \cdot \rangle|_{\got m_0}$ is Riemannian.

    Let $\got m_1 = \got m_0 \oplus \R E = (\R E)^{\perp}$. For any $D \in \got h^{\aut}$, $D(\got z) \subset \got z$ and $D(\got n') \subset \got n'$. It follows that $(\got n')^{\perp} \cap \got z = \R E$ is invariant under $D$, hence $\got h^{\aut}$-invariant. Therefore $\got m_1$ is $\got h^{\aut}$-invariant.

    Let $p: \got n \to \got m_0$ be the projection, and let $D \in \got h^{\aut}$. For every $A \in \got m_0 \subset \got m_1$, $D A \in \got m_1$. Hence
    \begin{equation*}
        D A = \tilde{D} A + \langle D A, F \rangle E,
    \end{equation*}
    where $\tilde{D} = p \circ D: \got m_0 \to \got m_0$.

    Set $\lambda_D(A) = \langle D A, F \rangle$. Then $\lambda_D: \got m_0 \to \R$ is a linear functional, and since $\got n'$ is $D$-invariant, $\lambda_D(A) = 0$ for every $A \in \got n'$. In particular,
    \begin{equation*}
        \lambda_D([A,B]) = 0 \quad \text{for every } A, B \in \got m_0.
    \end{equation*}

    We show that $\tilde D$ is skew-symmetric with respect to the restricted metric $\langle \cdot, \cdot \rangle|_{\got m_0}$. Let $A, B \in \got m_0$. Then
    \begin{equation*}
        \langle \tilde D A, B \rangle = \langle D A - \lambda_D(A) E, B \rangle = \langle D A, B \rangle = -\langle A, D B \rangle = -\langle A, D B + \lambda_D(B) E \rangle = -\langle A, \tilde D B \rangle.
    \end{equation*}
    Furthermore,
    \begin{equation*}
        \begin{split}
            \tilde D ([A,B]) &= D([A,B]) - \lambda_D([A,B]) E = D([A,B]) = [D A, B] + [A, D B] \\
                             &= [\tilde D A + \lambda_D(A) E, B] + [A, \tilde D B + \lambda_D(B) E] \\
                             &= [\tilde D A, B] + [A, \tilde D B].
        \end{split}
    \end{equation*}
    Hence $\tilde D \in \got h^{\aut}_0$, where $\got h^{\aut}_0$ denotes the Lie algebra of skew-symmetric derivations of $\got m_0$. Therefore $\tilde D \in \got h^{\aut}_0$ for every $D \in \got h^{\aut}$.

    Since $N$ is naturally reductive, there exists a linear and $\ad(\got h^{\aut})$-equivariant geodesic graph $\xi: \got n \to \got h^{\aut}$. Then the map $\tilde \xi: \got m_0 \to \got h^{\aut}_0$ defined by
    \begin{equation*}
        \tilde \xi(A) = p\bigl(\xi(A)|_{\got m_0}\bigr), \quad \text{for } A \in \got m_0,
    \end{equation*}
    is a geodesic graph for $\got m_0$. Indeed, $\tilde \xi$ is linear, and for every $A, B \in \got m_0$,
    \begin{equation*}
        \begin{split}
            \langle [A + \tilde \xi(A), B], A \rangle &= \langle [A,B], A \rangle + \langle \tilde \xi(A) B, A \rangle = \langle [A,B], A \rangle + \langle \xi(A) B - \lambda_{ \xi(A)}(B) E, A \rangle \\
            &= \langle [A + \xi(A), B], A \rangle = 0.
        \end{split}
    \end{equation*}

    Since $\got m_0$ is Riemannian with the restricted metric $\langle \cdot, \cdot \rangle|_{\got m_0}$, from \cite[Lemma 3] {ChenNikolayevskyNikonorov2022} it is naturally reductive.
\end{proof}

Obtaining a converse of Theorem~\ref{teo:central_extension_semidirect} is in general difficult. However, when $\got m_0$ is abelian we obtain:

\begin{theorem}\label{teo:double_extension_abelian}
    Let $(\got m_0, [\cdot,\cdot]_0)$ be an abelian Lie algebra equipped with a Riemannian metric $\langle \cdot, \cdot \rangle_0$. Let $\rho$ be a $2$-step nilpotent linear map of $\got m_0$ such that $\ker(\rho)=\mathrm{Im}(\rho)$ and let $E, F \notin \got m_0$. Set
    \begin{equation*}
        \got n = (\got m_0 \rtimes_\rho \R F) \times \R E,
    \end{equation*}
    and define a metric $\langle \cdot, \cdot \rangle$ on $\got n$ by
    \begin{equation*}
        \langle \cdot, \cdot \rangle|_{\got m_0} = \langle \cdot, \cdot \rangle_0, \quad \langle E, \got m_0 \rangle = \langle F, \got m_0 \rangle = \langle E, E \rangle = \langle F, F \rangle = 0, \quad \langle E, F \rangle = 1.
    \end{equation*}
    Then $\got n = \got n(\got g, \got v, \pi)$, where $(\got g, \got v, \pi)$ is a degenerate Lorentzian data set.
\end{theorem}

\begin{proof}
    The non-zero Lie brackets of $\got n$ are of the form
    \begin{equation*}
        [F, A] = \rho(A), \quad \text{for } A \in \got m_0.
    \end{equation*}
    Hence $\got n'=\mathrm{Im}(\rho)=\ker(\rho)$. Let $\got b$ be the orthogonal complement of $\got n' = \ker \rho$ in $\got m_0$. Then $\got n$ decomposes as the orthogonal direct sum
    \begin{equation*}
        \got n = \got n' \oplus \got b \oplus \operatorname{span}\{E, F\}.
    \end{equation*}
     
    Notice that $\rho|_{\got b}:\got b\rightarrow\got n'$ is an isomorphism. Indeed, it is injective since $\ker(\rho|_{\got b}) = \ker\rho\cap\got b = \got n'\cap\got b = \{0\}.$ Moreover, since $\got m_0=\got n'\oplus\got b$ and $
    \rho(\got n')=0$, we have $\got n' = \operatorname{Im}\rho = \rho(\got b)$,  so $\rho|_{\got b}$ is surjective. In particular,
    \begin{equation*}
        \dim \got b = \dim \got n'.
    \end{equation*}
    The subspaces $\got b$ and $\got n'$ are Riemannian. Let $\{X_1, X_2, \ldots, X_p\}$ be an orthonormal basis of $\got b$, and let $\{Z_1, Z_2, \ldots, Z_p\}$ be an orthonormal basis of $\got n'$.
    
    Let $\got v = \got b \oplus \operatorname{span}\{E, F\}$. Define a linear map $j: \got n' \to \got{so}(\got v, \langle \cdot, \cdot \rangle|_{\got v})$ by requiring
    \begin{equation}\label{eq:cond_j_ext_doble}
        \langle j(Z) X, Y \rangle = \langle [X,Y], Z \rangle \quad \text{for } X, Y \in \got v, \, Z \in \got n'.
    \end{equation}
    
    Let $Z \in \got n'$ and $Y \in \got b$, and write
    \begin{equation*}
        j(Z) Y = \sum_{i=1}^p \alpha_i X_i + \beta E + \gamma F
    \end{equation*}
    for some $\alpha_1, \ldots, \alpha_p, \beta, \gamma \in \R$. Then~\eqref{eq:cond_j_ext_doble} holds if $\alpha_i = 0$ for every $i = 1, 2, \ldots, p$, $\gamma = 0$, and
    \begin{equation*}
        \beta = -\langle \rho(Y), Z \rangle.
    \end{equation*}
    Therefore
    \begin{equation}\label{eq:cond_j_ext_doble_2}
        j(Z) Y = -\langle \rho(Y), Z \rangle E.
    \end{equation}
    
    For the vector $E$, $\langle j(Z) E, X \rangle = \langle [E, X], Z \rangle = 0$ for every $X \in \got v$. Hence,
    \begin{equation}\label{eq:cond_j_ext_doble_3}
        j(Z) E = 0.
    \end{equation}
    
    Next, consider $F$. Write
    \begin{equation*}
        j(Z) F = \sum_{i=1}^p \alpha_i X_i + \beta E + \gamma F.
    \end{equation*}
    Then~\eqref{eq:cond_j_ext_doble} holds if $\beta = \gamma = 0$ and
    \begin{equation*}
        \alpha_i = \langle \rho(X_i), Z \rangle \quad \text{for every } i = 1, 2, \ldots, p.
    \end{equation*}
    Hence,
    \begin{equation}\label{eq:cond_j_ext_doble_4}
        j(Z) F = \sum_{i=1}^p \langle \rho(X_i), Z \rangle X_i.
    \end{equation}
    
    Therefore equations~\eqref{eq:cond_j_ext_doble_2},~\eqref{eq:cond_j_ext_doble_3} and~\eqref{eq:cond_j_ext_doble_4} define a map $j:\got n'\rightarrow\got{so} (\got v,\langle\cdot,\cdot\rangle|_{\got v})$ satisfying~\eqref{eq:cond_j_ext_doble}. 
       
    We now verify that $j$ is a faithful representation and that its common kernel is precisely $\R E$.
    
    First, $j$ is injective. Indeed, suppose that $j(Z)=0$ for some $Z\in\got n'$. From \eqref{eq:cond_j_ext_doble_4}, we obtain that $\langle\rho(X_i),Z\rangle=0$ for every $i=1,\ldots,p$. Since $\rho|_{\mathfrak b}:\mathfrak b\to\mathfrak n'$ is an isomorphism, the vectors $\rho(X_1),\ldots,\rho(X_p)$ span $\mathfrak n'$. Hence $Z$ is orthogonal to $\mathfrak n'$, and therefore $Z=0$. Thus $j$ is injective. 
    
    We next show that
    \begin{equation*}
        [j(Z),j(Z')]=0 \quad\text{for every }Z,Z'\in\mathfrak n'.
    \end{equation*}
    For $Y\in\mathfrak b$, equation \eqref{eq:cond_j_ext_doble_2} and  the fact that $j(Z)E=0$ give
    \begin{equation*}
        j(Z)j(Z')Y=0=j(Z')j(Z)Y.
    \end{equation*}
    The same is clearly true on $\R E$. On the other hand, by \eqref{eq:cond_j_ext_doble_4} and \eqref{eq:cond_j_ext_doble_2},
    \begin{align*}
    j(Z)j(Z')F &= j(Z)\left( \sum_{i=1}^p \langle\rho(X_i),Z'\rangle X_i \right)= -\sum_{i=1}^p \langle\rho(X_i),Z'\rangle \langle\rho(X_i),Z\rangle E\\
        &= j(Z')j(Z)F.
    \end{align*}
    Consequently,
    \begin{equation*}
        [j(Z),j(Z')]=0.
    \end{equation*}
    Since $\got n'$ is abelian, $j$ is therefore a faithful representation of $\got n'$ on $\mathfrak v$.
    
    It remains to determine its common kernel. Decompose $A\in \got v$ as 
    \begin{equation*}
        A=Y+aE+bF\in\got v, \quad Y\in\got b,
    \end{equation*}
    and suppose that $j(Z)A=0$ for every  $Z\in\got n'$. Using \eqref{eq:cond_j_ext_doble_2}, \eqref{eq:cond_j_ext_doble_3} and \eqref{eq:cond_j_ext_doble_4} we obtain
    \begin{equation*}
        0 = j(Z)A = -\langle \rho(Y), Z \rangle E + b\sum_{i=1}^p \langle\rho(X_i), Z \rangle X_i
    \end{equation*}
    for every $Z\in\got n'$. Since $\got v = \got b\oplus\operatorname{span}\{E,F\}$, both the $\got b$-component and the $E$-component must vanish.
    
    If $b\neq0$, the vanishing of the $\got b$-component would imply
    \begin{equation*}
        \langle\rho(X_i),Z\rangle=0 \quad \text{for all }i\text{ and all }Z\in\got n',
    \end{equation*}
    which is impossible since the vectors $\rho(X_i)$ span $\got n'$. Hence $b=0$.
    
    Hence $\langle\rho(Y),Z\rangle=0$ for every $Z\in\got n'$. Since $\rho(Y)\in\got n'$, it follows that $\rho(Y)=0$.  But $Y\in\got b$ and $\rho|_{\got b}$ is injective, so $Y=0$. Hence $A=aE$. Conversely, \eqref{eq:cond_j_ext_doble_3} gives $j(Z)E=0$ for every $Z\in\got n'$. Therefore
    \begin{equation*}
        \bigcap_{Z\in\got n'}\ker j(Z)=\R E.
    \end{equation*}
    
    It follows that $(\got n',\got v,j)$ is a degenerate Lorentzian data set. Finally, by the definition of $j$ in \eqref{eq:cond_j_ext_doble}, the Lie bracket induced by this data set coincides with the original Lie bracket of $\got n$. Therefore $\got n=\got n(\got n',\got v,j)$, which completes the proof.
\end{proof}

The previous theorem yields a family of examples of Lie algebras $\got n$ associated with a degenerate Lorentzian data set. We illustrate this with an example which will be useful later (cf. Remark \ref{obs:novalidodegdataset}).

\begin{example}\label{example:1}
    Let $\got m_0$ be a $4$-dimensional abelian Lie algebra with basis $\{X_1, X_2, X_3, X_4\}$, endowed with a Riemannian metric $\langle \cdot, \cdot \rangle_0$ whose matrix with respect to this basis is the identity. Let $\rho: \got m_0 \to \got m_0$ be the linear map defined by
    \begin{equation*}
        \rho X_1 = X_2, \quad \rho X_3 = X_4, \quad \rho X_2 = \rho X_4 = 0.
    \end{equation*}
    Let $\got n = \operatorname{span}\{X_1, X_2, X_3, X_4, E, F\}$ with non-zero brackets
    \begin{equation*}
        [F, X_1] = X_2, \quad [F, X_3] = X_4.
    \end{equation*}
    Hence $\got n=(\got m_0\ltimes_{\rho}\mathbb R F)\times \mathbb R E$.

    The commutator of $\got n$ is $\got n' = \operatorname{span}\{X_2, X_4\}=\mathrm{Im}(\rho)=\ker(\rho)$, and its orthogonal complement is $\got v = \operatorname{span}\{X_1, X_3, E,F\}$. The representation $\pi: \got n' \to \got{so}(\got v, \langle \cdot, \cdot \rangle|_{\got v})$ acts on $\got v$, with respect to the basis $\{X_1, X_3, E, F\}$, as
    \begin{equation*}
        \pi(X_2) = \begin{pmatrix}
            0 & 0 & 0 & 1 \\
            0 & 0 & 0 & 0 \\
            -1 & 0 & 0 & 0 \\
            0 & 0 & 0 & 0
        \end{pmatrix}, \quad \pi(X_4) = \begin{pmatrix}
            0 & 0 & 0 & 0 \\
            0 & 0 & 0 & 1 \\
            0 & -1 & 0 & 0 \\
            0 & 0 & 0 & 0
        \end{pmatrix}.
    \end{equation*}

    Since $[\pi(X_2), \pi(X_4)] = 0$, the function $\tau: \got n' \to \got{so}(\got n')$ defined by $\tau_Z \equiv 0$ for each $Z \in \got n'$ satisfies item~\eqref{lemma:nat_red_cond_item_3} of Lemma~\ref{lemma:nat_red_cond}. Therefore $\got n = \got n(\got g, \got v, \pi)$, where $\got g = \got n'$ is an abelian Lie algebra.
\end{example}

\begin{observation}\label{rem:NWOLF}
The structure obtained in Theorem \ref{teo:central_extension_semidirect} is closely related to the double-extension description of Lorentzian geodesic orbit nilmanifolds given in \cite[Theorem 3]{NikolayevskyWolf2022}. More precisely, in \cite{NikolayevskyWolf2022} it is proved that, when the restriction of the metric to the commutator $\got n'$ is degenerate, the metric Lie algebra $\got n$ is a Lorentzian double extension of the metric Lie algebra of a Riemannian geodesic orbit nilmanifold. Thus, there exist lightlike vectors $E,F\in\got n$, with $\langle E,F\rangle=1$, and a Riemannian metric Lie algebra $\got m_0$ such that
\begin{equation*}
    \got n=\mathbb RF\oplus\got m_0\oplus\mathbb RE,
\end{equation*}
where $\mathbb RE$ is central in $\got m_0\oplus\mathbb RE$. Moreover, if $\got n$ is naturally reductive with respect to
$\Iso^{\aut}(N)$, in particular it is a geodesic orbit manifold and then the argument in the proof of \cite[Theorem~3]{NikolayevskyWolf2022} can be combined with the same projection argument used in the proof of Theorem \ref{teo:central_extension_semidirect}. Namely, a linear geodesic graph for $\got n$ induces, by projection, a linear geodesic graph for the Riemannian metric Lie algebra $\got m_0$. It then follows from \cite[Lemma 3]{ChenNikolayevskyNikonorov2022} that $\got m_0$ is naturally reductive.
\end{observation}

\section{Action of the representation \texorpdfstring{$\pi$}{pi} on \texorpdfstring{$\got v$}{v}}\label{sec:6_Action_of_the_representation_pi_on_v}

Let $(\got g, \got v, \pi)$ be a Lorentzian data set. When $\got g$ is Lorentzian and $\got v$ is Riemannian, we can obtain a useful decomposition of $\got v$ into $\pi(\got g)$-invariant subspaces with the same properties of Riemannian data sets (cf.~\cite[Lemma~3.11]{Lauret1999}). This result is proved almost in the same way as the Riemannian case. However there is a small gap in the statement and the  proof of \cite[Lemma 3.11]{Lauret1999}, so we include a more general statement and its proof here for the sake of completeness:

\begin{lemma}\label{lemma:central-action}
    Let $\got g$ be a Lie algebra and let $\pi:\got g\rightarrow\got{so}(V)$ be a representation on a finite-dimensional Euclidean vector space $V$. Let $\got c$ be a subspace of $\got z(\got g)$, and set
    \begin{equation*}
        W_0:=\bigcap_{Z\in\got c}\ker\pi(Z).
    \end{equation*}
    Then there exists an orthogonal $\pi(\got g)$-invariant decomposition
    \begin{equation*}
        V=W_0\oplus V_1\oplus\cdots\oplus V_k,
    \end{equation*}
    where $V_1,\ldots,V_k$ are irreducible subspaces.
    
    Moreover, for every $i=1,\ldots,k$, there exist $J_i\in\got{so}(V_i)$ such that $J_i^2 = -\Id_{V_i}$ and a non-zero linear functional $\lambda_i:\got c \rightarrow \R$ such that
    \begin{equation*}
        \left.\pi(Z)\right|_{V_i} = \lambda_i(Z)J_i
    \end{equation*}
    for every $Z\in\got c$.
    
    In particular, $\got c$ acts trivially on $W_0$ and non-trivially on every $V_i$, $i\geq1$.
\end{lemma}

\begin{proof}
    We first observe that since $\got c\subseteq \got z(\got g)$ and $\pi$ is a representation, then $[\pi(X),\pi(Z)]=0$ for each $X\in \got g$ and $Z\in \got c$, and so $W_0$ is $\pi(\got g)$-invariant.
    
    Since $\pi(\got g)\subset\got{so}(V)$, the orthogonal complement $W_0^\perp$ is also $\pi(\got g)$-invariant. Moreover, every invariant subspace of an orthogonal representation has an invariant orthogonal complement. Thus, by induction on the dimension, $W_0^\perp$ decomposes as an orthogonal direct sum of irreducible $\pi(\got g)$-subspaces:
    \begin{equation*}
        W_0^\perp = V_1\oplus\cdots\oplus V_k.
    \end{equation*}
    Therefore $V = W_0\oplus V_1\oplus\cdots\oplus V_k.$
    
    Notice that $\got c$ acts non-trivially on every $V_i$. Indeed, if $\left.\pi(Z)\right|_{V_i}=0$ for every $Z\in\got c$, then $V_i\subset W_0$. Since  $V_i\subset W_0^\perp$, it must be $V_i=\{0\}$, a contradiction.
    
    Fix now $i\in\{1,\ldots,k\}$. Since the action of $\got c$ on $V_i$ is non-trivial, there exists $Z_0\in\got c$ such that
    \begin{equation*}
        A:=\left.\pi(Z_0)\right|_{V_i}\neq 0.
    \end{equation*}
    Since $Z_0$ is central, $A$ commutes with $\left.\pi(\got g)\right|_{V_i}$. Consequently, $\ker A$ is a $\pi(\got g)$-invariant subspace of $V_i$. Since $V_i$ is irreducible and $A\neq 0$, it must be  $\ker A=\{0\}$. 
    
    Now $A$ is skew-symmetric, so $A^2$ is symmetric. Moreover, $A^2$ commutes with $\left.\pi(\got g)\right|_{V_i}$. Therefore every eigenspace of $A^2$ is $\pi(\got g)$-invariant. Since $V_i$ is irreducible, $A^2$ has only one eigenvalue. Thus 
    \begin{equation*}
        A^2=\alpha I
    \end{equation*}
    for some $\alpha\in\R$. Since $A$ is skew-symmetric, non-singular on $V_i$ and $V_i$ is Euclidean, we have
    \begin{equation*}
        \langle A^2v, v \rangle = -\langle Av, Av \rangle < 0
    \end{equation*}
    for every $v\in V_i$, $v\neq 0$. Hence $\alpha \| v \|^2 < 0$ for every $v \neq 0$, and therefore $\alpha < 0$. Hence we can choose $a>0$ such that $A^2 = -a^2I$. Define
    \begin{equation*}
        J_i:=\frac1a A.
    \end{equation*}
    Then $J_i\in\got{so}(V_i)$ and $J_i^2=-I$. 
    
    Let now $Z\in\got c$ and put $B:=\left.\pi(Z)\right|_{V_i}.$ Since $\got c$ is central, $B$ commutes with $A$, and hence
    with $J_i$. Consider
    \begin{equation*}
        C:=-J_iB.
    \end{equation*}
    Since $J_i$ and $B$ are skew-symmetric endomorphisms, $C$ is symmetric. Furthermore, since both $J_i$ and $B$ commute with $\left.\pi(\got g)\right|_{V_i}$, so does $C$. Therefore every eigenspace of $C$ is $\pi(\got g)$-invariant. By irreducibility,
    \begin{equation*}
        C = \lambda_i(Z)I
    \end{equation*}
    for some $\lambda_i(Z)\in\R$. It follows that $B=\lambda_i(Z)J_i$, that is,
    \begin{equation*}
        \left.\pi(Z)\right|_{V_i} = \lambda_i(Z)J_i.
    \end{equation*}
    
    Since $\pi$ is linear and $J_i$ is fixed and non-zero, the map $\lambda_i:\got c\rightarrow\R$ is linear. It is non-zero because the action of $\got c$ on $V_i$ is non-trivial.
\end{proof}

As a consequence of Lemma \ref{lemma:central-action} we immediately obtain the following generalization of \cite[Lemma 3.11]{Lauret1999} for Lorentzian data sets (observe that in \cite[Lemma 3.11]{Lauret1999}, the subspace $W_0$ should be added to the given decomposition). 

\begin{theorem}\label{lema:tec}
    Let $(\got g, \got v, \pi)$ be a Lorentzian data set with $\got v$ Riemannian, and write $\got g = \got g' \oplus \got c$, where $\got g' = [\got g, \got g]$ and $\got c$ is the center of $\got g$. Then $\got v$ admits an orthogonal, $\pi(\got g)$-invariant decomposition
    \begin{equation}\label{eq:descgotv}
        \got v = \got w_0\oplus\got v_1 \oplus \cdots \oplus \got v_k
    \end{equation} 
    where $\got w_0=\bigcap_{Z\in \got c}\ker(\pi(Z))$ and each $\got v_i$, $i = 1, \ldots, k$, is a $\pi(\got g)$-irreducible subspace such that there exists a skew-symmetric map $J_i: \got v_i \to \got v_i$ satisfying $J_i^2 = -I$, and a non-zero linear functional $\lambda_i:\got c\to \R$ such that  for every $Z \in \got c$,
    \begin{equation*}
        \pi(Z)|_{\got v_i} = \lambda_i(Z) J_i.
    \end{equation*}
\end{theorem}

When $\got g$ is Riemannian and $\got v$ is Lorentzian, it is not possible to decompose $\got v$ into the common kernel of $\pi(Z)$, $Z\in \got c$ and $\pi(\got g)$-irreducible orthogonal subspaces. We will show, however, that $\got v$ decomposes as an orthogonal sum of a first reducible factor, namely a $2$-dimensional Lorentzian subspace generated by two invariant lightlike vectors, and a sum of irreducible Riemannian subspaces (cf.~Theorem~\ref{teo:imp1} below). To do this, we first prove some technical results on the Lie algebra $\got{so}(1,n)$ of the Lorentzian isometry group.

Recall that if $\got v$ is Lorentzian, say of dimension $n+1$, then $\got v$ can be identified with the Lorentzian space $\R^{1,n}$, i.e., $\R^{n+1}$ with the canonical Lorentzian metric
\begin{equation*}
    \langle x, y \rangle_1 = -x_1 y_1 + \sum_{j=2}^{n+1} x_j y_j = x^t M y, \quad \text{with } M = \left(\begin{array}{c|c}
        -1 & 0 \\ \hline
        0 & \Id_n
    \end{array}\right),
\end{equation*}
and $\got{so}(\got v, \langle \cdot, \cdot \rangle_{\got v}) \simeq \got{so}(1,n)$, where $\got{so}(1,n)$ is the Lie algebra of the isometry group $\OO(1,n)$ of $(\R^{1,n}, \langle \cdot, \cdot \rangle_1)$:
\begin{align*}
    \got{so}(1,n) &= \{ A \in \got{gl}(n+1, \R) : \langle A v, w \rangle_1 + \langle v, A w \rangle_1 = 0 \text{ for all } v, w \in \R^{n+1} \} \\
                  &= \left\{ \left( \begin{array}{c|c}
                                0 & x^t \\ \hline
                                x & B
                            \end{array} \right) : x \in \R^n,\, B \in \got{so}(n) \right\}.
\end{align*}

\begin{lemma}\label{lema:SO_fixedvector}
    Let $K$ be a compact subgroup of the Lie group $\SO_+(1,n)$ (the connected component of the identity in $\OO(1,n)$). Then there exists a timelike vector $v \in \R^{1,n}$ that is fixed by every element of $K$.
\end{lemma}

\begin{proof}
    Consider the $n$-dimensional hyperbolic space $\mathbb{H}^n$, viewed as the $n$-dimensional Riemannian submanifold of $\R^{1,n}$
    \begin{equation*}
        \mathbb{H}^n = \{ x \in \R^{1,n} : \langle x, x \rangle = -1, \ x_1 > 0 \}.
    \end{equation*}
    Then $\SO_+(1,n)$ is the connected component of the identity of $\Iso(\mathbb{H}^n)$. Hence $K \subset \SO_+(1,n)$ is a compact group acting by isometries on $\mathbb{H}^n$. Since $\mathbb{H}^n$ is complete, simply connected and has negative sectional curvature, Cartan's Fixed Point Theorem (cf.~\cite[Theorem~1.4.6]{Eberlein1996}) gives that $K$ has a fixed point in $\mathbb{H}^n$, which is a timelike vector of the Lorentzian space $\R^{1,n}$.
\end{proof}

\begin{lemma}\label{lemma:decom_lemma_2}
    Let $\got s$ be a compact semisimple Lie subalgebra of $\got{so}(1,n)$. Then
    \begin{equation*}
        \got v_0 = \bigcap_{A \in \got s} \ker A
    \end{equation*}
    contains at least one timelike vector. In particular, $\got v_0$ is non-degenerate, and if $\dim \got v_0 \geq 2$, then $\got v_0$ is a Lorentzian space.
\end{lemma}

\begin{proof}
    Let $G$ be a connected subgroup of $\SO_+(1,n)$ with Lie algebra $\got s$. Since $\got s$ is compact and semisimple, $G$ is compact. By Lemma~\ref{lema:SO_fixedvector}, there is a timelike vector $v \in \R^{1,n}$ such that $e^{tA}(v) = v$ for every $A \in \got s$ and every $t$. Hence $A \cdot v = 0$ for every $A \in \got s$, i.e., $v \in \got v_0$. Therefore $\got v_0$ is either the one-dimensional subspace generated by $v$ or a Lorentzian subspace of $\got v$.
\end{proof}

The following result is immediate from the previous lemma.

\begin{corollary}\label{corolario_1}
    Let $\got s$ be a compact semisimple Lie algebra. Then there are no faithful representations
    \begin{equation*}
        \rho: \got s \to \got{so}(1,n)
    \end{equation*}
    without trivial subrepresentations.
\end{corollary}

\begin{corollary}\label{coro:corolario_2}
    Let $(\got g, \got v, \pi)$ be a Lorentzian or degenerate Lorentzian data set. Then $\got g$ is not semisimple; that is, its center $\got c$ is non-trivial.
\end{corollary}

\begin{proof}
    Consider first the case where $(\got g, \got v, \pi)$ is a Lorentzian data set. Suppose $\got g$ is semisimple, and decompose $\got g$ as the direct sum
    \begin{equation*}
        \got g = \got h_1 \oplus \cdots \oplus \got h_n
    \end{equation*}
    of simple ideals. By Proposition~\ref{prop:gcompact}, $\got g$ is compact and hence each $\got h_i$ is simple, compact, and $\dim \got h_i \geq 3$ for every $i = 1, \ldots, n$. It is standard that $\got h_i \perp \got h_j$ for $i \neq j$ with respect to the $\ad$-invariant metric $\langle \cdot, \cdot \rangle_{\got g}$, and $\langle \cdot, \cdot \rangle_{\got g}$ decomposes as
    \begin{equation*}
        \langle \cdot, \cdot \rangle_{\got g} = \lambda_1 B_1 + \cdots + \lambda_n B_n,
    \end{equation*}
    where $B_i$ is the Killing form of $\got h_i$ (see, e.g.,~\cite{ContiDelBarcoRossi2024}). Since $\got h_i$ is compact, each $B_i$ is negative-definite, and so either $\got g$ is Riemannian or $\langle \cdot, \cdot \rangle_{\got g}$ has signature $\nu \geq 2$, which cannot occur. We conclude that $\got g$ must be Riemannian, and hence $\got v$ is Lorentzian. But then $\pi: \got g \to \got{so}(\got v) \simeq \got{so}(1,n)$ is a faithful representation without trivial subrepresentations, contradicting Corollary~\ref{corolario_1}. Therefore $\got g$ is not semisimple when $(\got g, \got v, \pi)$ is a Lorentzian data set.

    Consider now the case where $(\got g, \got v, \pi)$ is a degenerate data set. Suppose $\got c = 0$, in which case $\got g$ is a compact, semisimple Lie algebra; since $\pi: \got g \to \got{so}(\got v) \simeq \got{so}(1,n)$ is faithful, $\pi(\got g)$ is a compact, semisimple subalgebra of $\got{so}(1,n)$. Then $\bigcap_{Z \in \got g} \ker(\pi(Z))$ is non-degenerate by Lemma~\ref{lemma:decom_lemma_2}, contradicting Definition~\ref{defi:data_set_degenerate}. Therefore $\got g$ is not semisimple.
\end{proof}

\begin{lemma}\label{lema:veryimpor}
    Let $\got a$ be an abelian subalgebra of $\got{so}(1,n)$ such that $\bigcap_{A\in\got a}\ker A=\{0\}$. Then
    \begin{equation*}
        \R^{1,n}=\got v_0\oplus\got v_1\oplus\cdots\oplus\got v_\ell
    \end{equation*}
    is an orthogonal sum of $\got a$-invariant subspaces, where $\got v_i$ is Riemannian, irreducible and $\got a$ acts non-trivially on $\got v_i$ for every $i\geq1$, and $\got v_0$ is Lorentzian of dimension $2$. Moreover, $\got v_0$ is the direct sum of two invariant one-dimensional subspaces generated by lightlike vectors.
\end{lemma}

\begin{proof}
    We argue by induction on $n$. Suppose first that $n=1$. Since $\got{so}(1,1)$ is one-dimensional and $\bigcap_{A\in\got a}\ker A=\{0\},$ we have $\got a\neq\{0\}$ and hence $\got a=\got{so}(1,1).$ The two lightlike lines $\R(1,1)$ and $\R(-1,1)$ are $\got a$-invariant, and  $\R^{1,1}=\R(1,1)\oplus\R(-1,1)$. Thus the result holds for $n=1$.
    
    Assume now that $n\geq2$ and that the result holds in every Lorentzian space $\R^{1,k}$ with $k<n$. Let
    $\mathcal A$ be the connected abelian Lie subgroup of $\SO^+(1,n)$ with Lie algebra $\got a$.
    
    Since $\SO^+(1,n)$ is non-abelian, $\mathcal A\subsetneq\SO^+(1,n)$. By \cite[Theorem 1.1]{DiScalaOlmos2001}, no connected proper subgroup of $\SO^+(1,n)$ acts irreducibly on $\R^{1,n}$. Hence there exists a non-zero proper $\mathcal A$-invariant subspace $V\subset\R^{1,n}$. 
    
    We first consider the case in which $V$ is non-degenerate. If $V$ is Riemannian, replace it by $V^\perp$, which is Lorentzian and
    $\mathcal A$-invariant. Thus we may assume that $V$ is Lorentzian. Then
    \begin{equation*}
        \R^{1,n}=V\oplus V^\perp
    \end{equation*}
    orthogonally, where $V^\perp$ is Riemannian, and both subspaces are $\got a$-invariant.
    
    Notice that the restriction of $\got a$ to each of these subspaces has trivial common kernel, otherwise one would have $\bigcap_{A\in\got a}\ker A\neq\{0\}$ against the hypothesis. 
    
    Since $V$ is a proper Lorentzian subspace of $\R^{1,n}$, the inductive hypothesis applies to $V$. Hence $V$ admits an orthogonal decomposition $V=\got v_0\oplus \got v_1\oplus\cdots\oplus\got v_r$, where $\got v_0$ is a two-dimensional Lorentzian subspace which is the sum of two invariant lightlike lines, and $\got v_i$ is Riemannian and irreducible for $i\geq1$.
    
    On the other hand, since $V^\perp$ is Riemannian, one has a representation $\pi:\got a\to \got{so}(V^{\perp})$, $\pi(a)=a_{|V^{\perp}}$ with $\got a=\got z(\got a)$. Since $\bigcap_{Z\in \got a}\ker(\pi(Z))=\{0\}$, applying Lemma \ref{lemma:central-action} one obtains an orthogonal decomposition of $V^{\perp}$ into irreducible $\got a$-invariant subspaces, $  V^\perp =\got v_{r+1}\oplus\cdots\oplus\got v_\ell$ on which $\got a$ acts non trivially and the lemma follows. 
    
    Assume now that the invariant subspace $V$ is degenerate. Since the ambient metric has index one, the radical of $V$ is one-dimensional. Thus $V$ contains a unique lightlike direction, say $\R w_0$. Since $\mathcal A$ preserves $V$ and acts by isometries, it preserves its radical. Hence $\mathcal A(\R w_0)=\R w_0$. Thus $w_0$ is a common eigenvector for all elements of $\mathcal A$. We claim that there exists $T\in\mathcal A$ such that
    \begin{equation*}
        Tw_0=\lambda w_0, \quad \lambda>0,\quad \lambda\neq1.
    \end{equation*}
    Indeed, for $A\in\got a$, write $e^{tA}w_0=\lambda_A(t)w_0$. Since $e^{tA}$ is invertible and $\lambda_A(0)=1$, we have $\lambda_A(t)>0$. If $\lambda_A(t)=1$ for every  $A\in\got a$ and every $t$, differentiating at $t=0$ gives $Aw_0=0$ for every $A\in\got a$, contradicting the hypothesis that $\bigcap_{A\in\got a}\ker A=\{0\}$. Thus there are $A_0\in\got a$ and $t_0\in\R$ such that, for $T=e^{t_0A_0}$, we have
    \begin{equation*}
        Tw_0=\lambda w_0
    \end{equation*}
    with $\lambda>0$ and $\lambda\neq1$. In particular, $\lambda\neq\pm1$.
    
    Since $T$ is a Lorentzian isometry and $\lambda\neq\pm1$, there exists a second lightlike eigenvector $w_1$ such that $\langle w_0,w_1\rangle\neq 0$ and 
    \begin{equation*}
        Tw_1=\lambda^{-1}w_1
    \end{equation*}
    (cf.~\cite[Lemma 1.61]{JavaloyesSanchez2010}). Set
    \begin{equation*}
        \got v_0:=\operatorname{span}\{w_0,w_1\}.
    \end{equation*}
    Then $\got v_0$ is a two-dimensional Lorentzian subspace (cf.~\cite[Lemma 1.44]{JavaloyesSanchez2010}).
    
    We shall see that $\got v_0$ is $\mathcal A$-invariant. We already know that $\R w_0$ is $\mathcal A$-invariant, so it remains
    to prove that $\R w_1$ is $\mathcal A$-invariant. Let $U:=\got v_0^\perp$. Then $U$ is Riemannian and $\R^{1,n}=\got v_0\oplus U.$ Let $E_{\lambda^{-1}}$ denote the eigenspace of $T$ corresponding to $\lambda^{-1}$. If $w\in E_{\lambda^{-1}}$, write
    \begin{equation*}
        w=aw_0+bw_1+u, \quad u\in U.
    \end{equation*}
    Since $Tw=\lambda^{-1}w$, comparison of the $w_0$-components gives  $\lambda a=\lambda^{-1}a$. As $\lambda\neq\pm1$, it follows that $a=0$. Moreover, $Tu=\lambda^{-1}u$. If $u\neq0$, then $u$ is a spacelike eigenvector of the Lorentzian isometry $T$ with eigenvalue $\lambda^{-1}\neq\pm1$, which is impossible by \cite[Proposition 1.57]{JavaloyesSanchez2010}. Hence $u=0$. Consequently,
    \begin{equation*}
        E_{\lambda^{-1}}=\R w_1.
    \end{equation*}
    
    Since $\mathcal A$ is abelian, every element of $\mathcal A$ commutes with $T$, and therefore preserves every eigenspace of
    $T$. In particular,
    \begin{equation*}
        \mathcal A(\R w_1)\subseteq\R w_1.
    \end{equation*}
    Hence both $\R w_0$ and $\R w_1$ are invariant, and therefore $\got v_0$ is $\got a$-invariant.
    
    It remains to decompose $U=\got v_0^\perp$. Since $\got v_0$ is non-degenerate and invariant, $U$ is Riemannian and $\got a$-invariant. Moreover, $\bigcap_{A\in\got a}\ker\bigl(A|_U\bigr)=\{0\}$, because a vector in this intersection would belong to $\bigcap_{A\in\got a}\ker A$.
    
    Therefore  the decomposition of $U$ given by Lemma \ref{lemma:central-action} applied to the canonical representation $\pi:\got a\to \got{so}(U)$, $a\mapsto a|_{U}$, has no trivial-center summand, and gives an orthogonal decomposition $U=\got v_1\oplus\cdots\oplus\got v_\ell$ into irreducible Riemannian $\got a$-invariant subspaces.
    
    Combining this decomposition with $\got v_0$ yields
    \begin{equation*}
        \R^{1,n} = \got v_0\oplus \got v_1\oplus\cdots\oplus\got v_\ell,
    \end{equation*}
    where $\got v_0$ is two-dimensional Lorentzian and is the direct sum of the two invariant lightlike lines $\R w_0$ and $\R w_1$, while every $\got v_i$, $i\geq1$, is Riemannian and irreducible.
\end{proof}

\begin{lemma}\label{lemma:tec_lemma_imp1}
    Let $(\got g, \got v, \pi)$ be a Lorentzian data set with $\got v$ Lorentzian, or a degenerate Lorentzian data set. If $\got g'=[\got g,\got g]$ is non-trivial, then the subspace of $\got v$ defined by
    \begin{equation*}
        \got p = \bigcap_{Z \in \got g'} \ker \pi(Z)
    \end{equation*}
    is $\pi(\got g)$-invariant, Lorentzian, and $\dim \got p \geq 2$.
\end{lemma}

\begin{proof}
    Observe first that since $\got g'$ is an ideal of $\got g$, then  $\got p$ is $\pi(\got g)$-invariant. Indeed, for any  $
    X\in\got p$, $ Y\in\got g$ and $ Z\in\got g'$, one has 
    \begin{equation*}
        \pi(Z)(\pi(Y)(X)) = \pi(Y)(\pi(Z)(X)) + \pi([Z,Y])(X) = 0.
    \end{equation*}
    
    We next show that $\got p$ is Lorentzian. Since $\got g'$ is compact and semisimple, $\pi(\got g')\subset\got{so}(\got v)$ is a compact semisimple Lie subalgebra. Moreover,
    \begin{equation*}
        \got p=\bigcap_{Z\in\got g'}\ker\pi(Z)=\bigcap_{A\in\pi(\got g')}\ker A.
    \end{equation*}
    Hence, by Lemma \ref{lemma:decom_lemma_2}, $\got p$ contains a timelike vector and is non-degenerate. Since $\got v$ is Lorentzian, the restriction of the metric to the non-degenerate subspace $\got p$ has index one. Therefore $\got p$ is a Lorentzian subspace of $\got v$.
    
    It remains to prove that $\dim\got p \geq 2$. Suppose, by contradiction, that $\dim\got p=1$. Since $\got p$ is Lorentzian, there exists a timelike vector $X_0\in \got v$ such that $\got p=\R X_0$. Since $\got p$ is $\pi(\got g)$-invariant, for every $Y\in\got g$ there exists $\lambda_Y\in\R$ such that
    \begin{equation*}
        \pi(Y)X_0=\lambda_YX_0.
    \end{equation*}
    Since $\pi(Y)\in\got{so}(\got v)$, we have
    \begin{equation*}
        0 = \langle \pi(Y)X_0, X_0 \rangle = \lambda_Y \langle X_0, X_0 \rangle.
    \end{equation*}
    Since $X_0$ is timelike, $\langle X_0,X_0\rangle<0$ and therefore $\lambda_Y=0$ for every $Y\in\got g$. Thus
    \begin{equation*}
        X_0\in \bigcap_{Y\in\got g}\ker\pi(Y).
    \end{equation*}
    
    If $(\got g,\got v,\pi)$ is a Lorentzian data set, then by definition $\bigcap_{Y\in\got g}\ker\pi(Y)=\{0\}$,  which contradicts $X_0\neq0$.
    
    If $(\got g,\got v,\pi)$ is a degenerate Lorentzian data set, then $\bigcap_{Y\in\got g}\ker\pi(Y)=\R E$, where $E$ is lightlike. Hence $X_0\in\R E$, which is impossible because $X_0$ is timelike.
    
    Therefore $\dim\got p\geq2$, and the proof is complete.
\end{proof}

We can now generalize Theorem~\ref{lema:tec} to the case where $(\got g, \got v, \pi)$ is a Lorentzian data set with Lorentzian $\got v$.

\begin{theorem}\label{teo:imp1}
    Let $(\got g,\got v,\pi)$ be a Lorentzian data set with $\got v$ Lorentzian, and let $\got c=\got z(\got g)$. Set $\got w_0=\bigcap_{Z\in\got c}\ker\pi(Z)$. Then:
    \begin{enumerate}
        \item The subspace $\got w_0$ is Riemannian and $\pi(\got g)$-invariant, and $\got v$ admits an orthogonal decomposition 
        \begin{equation*}
            \got v=\got w_0\oplus\got v_0\oplus\got v_1 \oplus\cdots\oplus\got v_{\ell}
        \end{equation*}
        into $\pi(\got g)$-invariant subspaces, where $\got v_i$ is Riemannian and irreducible for every $i\geq1$, and $\got v_0$ is Lorentzian of dimension $2$. Moreover, $\got v_0$ is the direct sum of two $\pi(\got g)$-invariant one-dimensional lightlike subspaces.
        
        \item\label{aux:item2teoimp}
            For every $i=1,\ldots,\ell$, there exist $J_i\in\got{so}(\got v_i)$ such that $J_i^2=-\Id_{\got v_i}$, and a non-zero linear functional $\lambda_i:\got c\to\R$ such that for every $Z\in\got c$,
            \begin{equation*}
                \left.\pi(Z)\right|_{\got v_i} = \lambda_i(Z)J_i.
            \end{equation*}
        
        \item\label{aux:item3teoimp}
            There exist $J_0\in\got{so}(\got v_0)$ such that $J_0^2=\Id_{\got v_0}$ and a non-zero linear functional $\lambda_0:\got c\to \R$ such that for every $Z\in\got c$,
            \begin{equation*}
                \left.\pi(Z)\right|_{\got v_0} = \lambda_0(Z)J_0.
            \end{equation*}
    \end{enumerate}
\end{theorem}

\begin{proof}
    Observe first that since $[\got g,\got c]=0$, it immediately follows that  $\got w_0$ is $\pi(\got g)$-invariant. We distinguish two cases: when $\got g$ is abelian or when $\got g'\neq 0$. 
    
    If $\got g$ is abelian, then $\got g=\got c$. Since $(\got g,\got v,\pi)$ is a Lorentzian data set, then $\got w_0= \bigcap_{Z\in\got c}\ker\pi(Z)=\{0\}$. Let
    \begin{equation*}
        \got a=\pi(\got g)\subset\got{so}(\got v).
    \end{equation*}
    Then $\got a$ is abelian and $\bigcap_{A\in\got a}\ker A=\{0\}$.  By Lemma \ref{lema:veryimpor} there is an orthogonal decomposition
    \begin{equation*}
        \got v= \got v_0 \oplus \got v_1\oplus\cdots\oplus\got v_{\ell},
    \end{equation*}
    where $\got v_0$ is Lorentzian of dimension $2$ and is the direct sum of two invariant lightlike lines, while $\got v_i$, $i\geq1$, is Riemannian and irreducible. Set $\tilde{\got v}=\got v_1\oplus\cdots\oplus \got v_{\ell}$. Then $\tilde{\got v}$ is $\pi(\got g)$-invariant, $\tilde{\got v}$ is Euclidean and we can consider a representation $\pi:\got g\to \got{so}(\tilde{\got v})$ which has $\got v_i$ as irreducible factors. Observe that for $i\geq1$, the action of $\got c=\got g$ on $\got v_i$ cannot be trivial, since otherwise $ \got v_i\subset \bigcap_{Z\in\got g}\ker\pi(Z)$, contrary to the defining property of a Lorentzian data set. Thus, by Lemma \ref{lemma:central-action}, there exist  $J_i\in \got{so}(\got v_i)$ and a non-zero linear functional $\lambda_i: \got c \to \R$ such that 
    \begin{equation*}
        \left.\pi(Z)\right|_{\got v_i}= \lambda_i(Z)J_i
    \end{equation*}
    for every $Z\in\got c$.
    
    On the other hand, since $\got{so}(\got v_0)\simeq\got{so}(1,1)$ is one-dimensional, we may choose a generator $J_0\in\got{so}(\got v_0)$ normalized by $J_0^2=\Id_{\got v_0}$. Then there exists a linear functional $\lambda_0 : \got c \to \R$ such that
    \begin{equation*}
        \left.\pi(Z)\right|_{\got v_0} = \lambda_0(Z)J_0.
    \end{equation*}
    Moreover, $\lambda_0\neq0$, for otherwise $\got v_0$ would be contained in the common kernel of $\pi(\got g)$. This proves the theorem when $\got g$ is abelian.
    
    Suppose now that $\got g$ is non-abelian, so  $\got g'\neq0$. Set
    \begin{equation*}
        \got p=\bigcap_{Z\in\got g'}\ker\pi(Z).
    \end{equation*}
    By Lemma \ref{lemma:tec_lemma_imp1}, $\got p$ is a $\pi(\got g)$-invariant Lorentzian subspace and $
    \dim\got p\geq2.$ Hence one can decompose $\got v = \got p\oplus\got p^\perp$ orthogonally, where $\got p^\perp$ is Riemannian and
    $\pi(\got g)$-invariant. We claim that 
    \begin{equation*}
        \got w_0\subset\got p^\perp.
    \end{equation*}
    Indeed, let $X\in\got w_0$ and decompose $X = A + B$, with $A\in\got p$ and $B\in\got p^\perp$. Since both $\got p$ and $\got p^\perp$ are
    $\pi(\got g)$-invariant, and for every $Z\in\got c$ we have $0=\pi(Z)X=\pi(Z)A+\pi(Z)B$, with the two terms belonging respectively to $\got p$ and $\got p^\perp$, we get that 
    \begin{equation*}
        \pi(Z)A = 0, \quad \pi(Z)B = 0
    \end{equation*}
    for every $Z\in \got c$. But also $\pi(Y)A=0$ for every $Y\in \got g'$. Hence $A\in \bigcap_{W\in \got g}\ker(\pi(W))=\{0\}$. We conclude that $X=B\in \got p^{\perp}$, proving the claim. In particular, $\got w_0$ is Riemannian. 
    
    So if we consider the representation $\left.\pi\right|_{\got p^\perp}: \got g\to\got{so}(\got p^\perp)$, the subspace on which the center acts trivially is
    \begin{equation*}
        \bigcap_{Z\in\got c} \ker\left( \left.\pi(Z)\right|_{\got p^\perp} \right).
    \end{equation*}
    Since $\got w_0\subset\got p^\perp$, this subspace is precisely $\got w_0$. Hence Lemma \ref{lemma:central-action} gives an orthogonal
    decomposition
    \begin{equation*}
        \got p^\perp=\got w_0\oplus\got v_{1}\oplus\cdots\oplus\got v_k,
    \end{equation*}
    where every $\got v_i$, $i=1,\ldots,k$, is Riemannian and irreducible under $\pi(\got g)$, the action of $\got c$ on each $\got v_i$ is non-trivial and item \eqref{aux:item2teoimp} is satisfied on each $\got v_i$. 
    
    We now study the action on $\got p$. Since $\got g'$ acts trivially on $\got p$, only the center acts non-trivially there. Consider
    \begin{equation*}
        \mu:\got c\to\got{so}(\got p), \qquad \mu(Z)=\left.\pi(Z)\right|_{\got p},
    \end{equation*}
    and set $\got a=\mu(\got c)$. Then  $\got a$ is an abelian subalgebra of $\got{so}(\got p)$. Moreover, $\bigcap_{A\in\got a}\ker A=\{0\}$. In fact, if $X\in\got p$ and $\pi(Z)X=0$ for every $Z\in\got c$, then, by the definition of $\got p$, $X$ is also annihilated by $\pi(\got g')$. Hence
    \begin{equation*}
        X\in \bigcap_{Y\in\got g}\ker\pi(Y) = \{0\}.
    \end{equation*}
    
    Lemma \ref{lema:veryimpor} therefore yields an orthogonal decomposition
    \begin{equation*}
        \got p=\got v_0 \oplus \got v_{k+1} \oplus\cdots\oplus \got v_{\ell},
    \end{equation*}
    where $\got v_0$ is Lorentzian of dimension $2$ and splits as the direct sum of two $\got a$-invariant lightlike lines, while $\got v_{k+1},\ldots,\got v_{\ell}$ are Riemannian and irreducible under the action of $\got a$.
    
    These subspaces are in fact $\pi(\got g)$-invariant. They are invariant under $\pi(\got c)$ by construction, and $\pi(\got g')$ vanishes identically on $\got p$. Moreover, for $k+1\leq i\leq \ell$, the action of $\got c$ on $\got v_i$ is non-trivial; otherwise $\got v_i \subset \bigcap_{Z\in\got c}\ker\mu(Z) = \{0\}.$ The existence of the operators $J_i$ and the functionals $\lambda_i$ satisfying items \eqref{aux:item2teoimp} and \eqref{aux:item3teoimp} follows in the same way as the abelian case. 
    
    Combining the decompositions of $\got p$ and $\got p^\perp$ gives
    \begin{equation*}
        \got v=\got w_0\oplus\got v_0\oplus\got v_1\oplus\cdots\oplus\got v_{\ell},
    \end{equation*}
    where items \eqref{aux:item2teoimp} and \eqref{aux:item3teoimp} are satisfied. 
\end{proof}

\begin{observation}\label{obs:novalidodegdataset}
    The conclusions of Theorem \ref{teo:imp1} do not extend, in general, to degenerate Lorentzian data sets. This phenomenon already occurs when $\got g$ is abelian, as shown by Example \ref{example:1}. Indeed, in that example,
    \begin{equation*}
        \bigcap_{Z\in\got g}\ker\pi(Z) = \R E,
    \end{equation*}
    so the representation is reducible. However, it has no non-trivial proper non-degenerate $\pi(\got g)$-invariant subspaces. i.e. it is weakly irreducible. Hence $\got v$ does not admit any non-trivial orthogonal decomposition into $\pi(\got g)$-invariant subspaces. Moreover, the common eigenvectors of the family $\{\pi(Z):Z\in\got g\}$ are precisely the non-zero multiples of $E$. The most one can have in the degenerate setting is the following result:
\end{observation}

\begin{theorem}
    Let $(\got g,\got v,\pi)$ be a degenerate Lorentzian data set and assume that $\got g'=[\got g,\got g]\neq\{0\}$. 
    Let $\got c=\got z(\got g)$ and set $\got p = \bigcap_{Z\in\got g'}\ker\pi(Z).$
    Then:
    
    \begin{enumerate}
        \item\label{item1aux}
        The subspace $\got p$ is Lorentzian, $\pi(\got g)$-invariant, and $\dim\got p\geq2$. Moreover, $\bigcap_{Z\in\got c} \ker\bigl(\pi(Z)|_{\got p}\bigr) = \R E,$ where $\R E$ is the common kernel of the degenerate Lorentzian data set.
        
        \item\label{item2aux}
        The orthogonal complement $\got p^{\perp}$ is Riemannian and $\pi(\got g)$-invariant. If  $\got w_0=\bigcap_{Z\in\got c} \ker\bigl(\pi(Z)|_{\got p^\perp}\bigr),$ then there exists an orthogonal decomposition
        \begin{equation*}
            \got v=\got p\oplus\got w_0\oplus\got v_1\oplus\cdots\oplus\got v_k
        \end{equation*}
        into $\pi(\got g)$-invariant subspaces, where $\got w_0$ is Riemannian and, for every $i=1,\ldots,k$, $\got v_i$ is Riemannian and irreducible.
        
        \item\label{item3aux}
        For every $i=1, \ldots, k$, there exist $J_i\in\got{so}(\got v_i)$, such that $J_i^2=-\Id_{\got v_i},$ and a non-zero linear functional $\lambda_i:\got c\to \R$ such that
        \begin{equation*}
            \left.\pi(Z)\right|_{\got v_i} =\lambda_i(Z)J_i
        \end{equation*}
        for every $Z\in\got c$.
    \end{enumerate}
\end{theorem}

\begin{proof}
    Since $\got g'\neq\{0\}$, one can apply Lemma \ref{lemma:tec_lemma_imp1} and obtain that $\got p= \bigcap_{Z\in\got g'}\ker\pi(Z)$ is a $\pi(\got g)$-invariant Lorentzian subspace and that $\dim\got p\geq2$. This proves the first assertion of item \eqref{item1aux}.
    
    We now determine the common kernel of the action of $\got c$ on $\got p$. Since $(\got g,\got v,\pi)$ is a degenerate Lorentzian data
    set, there exists a non-zero lightlike vector $E\in\got v$ such that $\bigcap_{Z\in\got g}\ker\pi(Z)=\R E.$
    In particular, $E$ is annihilated by $\pi(\got g')$, and hence $E\in\got p$. Also, $E$ is annihilated by $\pi(\got c)$, so
    \begin{equation*}
        \R E \subset \bigcap_{Z\in\got c} \ker\bigl(\pi(Z)|_{\got p}\bigr).
    \end{equation*}
    
    Conversely, suppose $X\in\got p$ satisfies $\pi(Z)X=0$ for every $Z\in\got c$. Since $X\in\got p$, we also have $\pi(Y)X=0$ for every $Y\in\got g$. We then conclude that $X\in \bigcap_{Z\in\got g}\ker\pi(Z) =\R E$.  Therefore
    \begin{equation*}
        \bigcap_{Z\in\got c} \ker\bigl(\pi(Z)|_{\got p}\bigr)= \R E.
    \end{equation*}
    This completes the proof of item \eqref{item1aux}.
    
    Since $\got p$ is non-degenerate and Lorentzian, we have the orthogonal decomposition $\got v= \got p\oplus\got p^\perp.$ Moreover, $\got p^\perp$ is Riemannian. Since $\got p$ is $\pi(\got g)$-invariant and every $\pi(Z)$ is skew-symmetric, $\got p^\perp$ is also $\pi(\got g)$-invariant. Consider the  representation
    \begin{equation*}
        \rho: \got g\to\got{so}(\got p^\perp), \quad \rho(Z)=\left.\pi(Z)\right|_{\got p^\perp}.
    \end{equation*}
    Set $\got w_0= \bigcap_{Z\in\got c}\ker\rho(Z).$  Since $\got p^{\perp}$ is Euclidean, one can apply Lemma \ref{lemma:central-action} to  $\rho$, and obtain an orthogonal decomposition
    \begin{equation*}
        \got p^\perp=\got w_0\oplus \got v_1 \oplus\cdots\oplus \got v_k,
    \end{equation*}
    such that items \eqref{item2aux} and \eqref{item3aux} hold. 
\end{proof}

\section{The isotropy algebra $\got h^{\aut}$}\label{sec:7_action_degenerate}

Let $(\got g, \got v, \pi)$ be a Lorentzian data set and let $N = N(\got g, \got v, \pi)$ be the simply connected $2$-step nilpotent Lorentzian Lie group with Lie algebra
\begin{equation*}
    \got n = \got n(\got g, \got v, \pi) = \got g \oplus \got v.
\end{equation*}
The Lie algebra of $\Iso^{\aut}(N) = N \rtimes H^{\aut}$ is
\begin{equation*}
    \got{iso}^{\aut}(N) = \got n \rtimes \got h^{\aut}.
\end{equation*}
Therefore, to obtain $\Iso^{\aut}(N)$, it suffices to compute $H^{\aut}$. It was proved in~\cite{DelBarcoOvando2014, Ovando2013} that
\begin{equation}\label{eq:H_aut}
    H^{\aut} = \{ (\phi, T) \in \OO(\got g, \langle \cdot, \cdot \rangle_{\got g}) \times \OO(\got v, \langle \cdot, \cdot \rangle_{\got v}) : \pi(\phi Z) = T \pi(Z) T^{-1} \text{ for every } Z \in \got g \},
\end{equation}
and that its Lie algebra is
\begin{equation}\label{eq:H_aut_lie}
    \got h^{\aut} = \{ (A, B) \in \got{so}(\got g, \langle \cdot, \cdot \rangle_{\got g}) \times \got{so}(\got v, \langle \cdot, \cdot \rangle_{\got v}) : [B, \pi(Z)] = \pi(A Z) \text{ for every } Z \in \got g \}.
\end{equation}

For a degenerate Lorentzian data set, the computation is identical and yields the same descriptions of $H^{\aut}$ and $\got h^{\aut}$. Therefore,~\eqref{eq:H_aut} and~\eqref{eq:H_aut_lie} remain valid for degenerate data sets.

In this section we give a simpler description of $\got h^{\aut}$ from that of \eqref{eq:H_aut_lie}. It is similar to that of the Riemannian case proved in~\cite[Theorem~3.12]{Lauret1999}. In~\cite{Ovando2013}, such a description was given for pseudo-Riemannian spaces under the assumption that $\got g$ is semisimple (cf.~the discussion after~\cite[Proposition~3.5]{Ovando2013}); however, as observed in Corollary~\ref{coro:corolario_2}, this is never the case when $\got n = \got n(\got g, \got v, \pi)$ with $(\got g, \got v, \pi)$ a Lorentzian or a degenerate Lorentzian data set.

We start by proving some common properties that hold both in the Lorentzian and degenerate Lorentzian cases. Then we shall treat both cases separately.  

\begin{lemma}\label{lemma:Aderivacion}
    Let $(\got g,\got v,\pi)$ be a Lorentzian data set or a degenerate Lorentzian data set and let $(A,B)\in\got h^{\aut}$. Then $A$ is a derivation of $\got g$. In particular, if $\got g=\got g'\oplus \got c$ with $\got c=\got z(\got g)$ and $\got g'=[\got g,\got g]$, $A(\got g')\subset\got g'$ and $ A(\got c)\subset\got c$.
\end{lemma}

\begin{proof}
    Let $Z,W\in\got g$. From \eqref{eq:H_aut_lie}, the fact that $\pi$ is a representation and the Jacobi identity in $\End(\got v)$, one gets
    \begin{align*}
    \pi(A[Z,W])
        &=[B,\pi([Z,W])]=[B,[\pi(Z),\pi(W)]]\\
        &=[[B,\pi(Z)],\pi(W)]+[\pi(Z),[B,\pi(W)]]\\
        &=[\pi(AZ),\pi(W)]+[\pi(Z),\pi(AW)]\\
        &=\pi([AZ,W]+[Z,AW]).
    \end{align*}
    Since $\pi$ is faithful, $A[Z,W]=[AZ,W]+[Z,AW]$, and so $A\in\operatorname{Der}(\got g)$. The last assertion is standard.
\end{proof}

\begin{proposition}\label{prop:h_aut_general}
    Let $(\got g,\got v,\pi)$ be a Lorentzian data set, or a degenerate Lorentzian data set with $\got g'\neq 0$. Then
    \begin{equation*}
        \got h^{\aut}=\got g'\oplus\got u, \quad [\got g',\got u]=0,
    \end{equation*}
    where $\got g'$ is identified with $\{(\ad (Z),\pi(Z)):Z\in\got g'\}\subset\got h^{\aut}$ and
    \begin{equation*}
        \got u=\{(A,B)\in\got h^{\aut}:A|_{\got g'}=0\}.
    \end{equation*}
    Equivalently, $(A,B)\in \got u$ if and only if $(A,B)\in \got{so}(\got c)\oplus \got{so}(\got v)$ and verifies
    \begin{equation*}
        \begin{cases}
            [B,\pi(Z)]=0,&Z\in\got g',\\
            [B,\pi(Z)]=\pi(AZ),&Z\in\got c.
        \end{cases}
    \end{equation*}
\end{proposition}

\begin{proof} 
    Let $(A,B)\in\got h^{\aut}$. Since $A|_{\got g'}$ is a derivation of the compact semisimple algebra $\got g'$, there exists a unique $Z_0\in\got g'$ such that
    \begin{equation*}
        A|_{\got g'}=\ad(Z_0)|_{\got g'}.
    \end{equation*}
    From \eqref{eq:H_aut_lie}, and the fact that $\pi$ is a representation, it follows immediately that  $(\ad Z_0,\pi(Z_0))\in\got h^{\aut}$.
    Thus
    \begin{equation*}
        (A_0,B_0):=(A-\ad (Z_0),B-\pi(Z_0))\in\got h_{\mathrm{aut}}, \quad A_0|_{\got g'}=0,
    \end{equation*}
    and therefore $\got h^{\aut}=\got g'+\got u$. The sum is direct. In fact, if $(\ad (Z),\pi(Z))\in\got u$ with $Z\in\got g'$, then $\ad (Z)|_{\got g'}=0$, and so $Z=0$ because $\got g'$ is semisimple.
    
    Finally, let $(A,B)\in\got u$ and $Z\in\got g'$. Since $A(Z)=0$, equation \eqref{eq:H_aut_lie} gives $[B,\pi(Z)]=0$. Also $A|_{\got g'}=0$, $A(\got c)\subset\got c$, and $\ad (Z)$ vanishes on $\got c$, hence $[A,\ad Z]=0$. Therefore $[(A,B),(\ad (Z),\pi(Z))]=(0,0),$ and so $[\got g',\got u]=0$.
\end{proof}

\begin{observation}
    Observe that Proposition \ref{prop:h_aut_general} also holds when $\got g'=0$. In this case we get $\got h^{\aut}=\got u$, which is no other than the decomposition given in equation \eqref{eq:H_aut_lie}.
\end{observation}

We shall now give a better description of $\got h^{\aut}$ for Lorentzian data sets. We start by showing that the kernel of any map $\pi(Z)$ for $Z \in \got c$ can be decomposed in accordance with the decompositions of $\got v$ given by Theorems \ref{lema:tec} and \ref{teo:imp1}, and as a consequence that $\ker \pi(Z)$ is always a non-degenerate subspace of $\got v$.

\begin{proposition}\label{prop:ker_nodeg}
    Let $(\got g,\got v,\pi)$ be a Lorentzian data set, put $\got g=\got g'\oplus \got c$ and let $Z\in\got c$. Then $\ker\pi(Z)$ is non-degenerate.
    
    Moreover, decompose 
    $\got v=\got w_0\oplus\got v_0\oplus\got v_1\oplus\cdots\oplus\got v_k$ as in Theorems \ref{lema:tec} and \ref{teo:imp1}, where if $\got v$ is Riemannian we use the convention $\got v_0=\{ 0 \}$. Then
    \begin{equation*}
        \ker\pi(Z)=\got w_0\oplus\bigoplus_{\lambda_i(Z)=0}\got v_i
    \end{equation*}
\end{proposition}

\begin{proof}
    First of all, one has $\pi(Z)|_{\got w_0}=0$, so $\got w_0\subseteq \ker(\pi(Z))$ and it is a Riemannian subspace of $\got v$. 
    On the other hand, for $i\ge1$,
    $\pi(Z)|_{\got v_i}=\lambda_i(Z)J_i$. Since $ J_i^2=-\Id$, 
    every $J_i$ is invertible. Hence the kernel on each $\got v_i$ is either $\{0\}$ or all of $\got v_i$. The case $i=0$ follows in the same way. Therefore
    \begin{equation*}
        \ker\pi(Z)=\got w_0\oplus\bigoplus_{\lambda_i(Z)=0}\got v_i
    \end{equation*}
    is an orthogonal direct sum of non-degenerate subspaces, and hence is non-degenerate. 
\end{proof}

Given a Lorentzian or degenerate Lorentzian data set $(\got g, \got v, \pi)$, we denote by $\End_{\pi}(\got v)$ the space of $\pi$-intertwining endomorphisms of $\got v$, i.e., the endomorphisms $B \in \End(\got v)$ such that
\begin{equation*}
    \pi(Z) (B(X)) = B(\pi(Z) X)
\end{equation*}
for all $Z \in \got g$ and $X \in \got v$, or equivalently,
\begin{equation*}
    [B, \pi(Z)] = 0 \quad \text{for all } Z \in \got g.
\end{equation*}

\begin{theorem}\label{thm:lorentzian-isotropy}
    Let $(\got g,\got v,\pi)$ be a Lorentzian data set. With the notation of Proposition \ref{prop:h_aut_general}, one has
    \begin{equation*}
        \got u=\End_{\pi}(\got v)\cap\got{so}(\got v) = \{B\in\got{so}(\got v):[B,\pi(Z)]=0\ \forall Z\in\got g\},
    \end{equation*}
    where $B$ is identified with $(0,B)\in \got h ^{\aut}$. Consequently, $\got h^{\aut} = \got g'\oplus\bigl(\End_{\pi}(\got v)\cap\got{so}(\got v)\bigr)$ and the two summands commute.
\end{theorem}

\begin{proof}
    Let $(A,B)\in\got u$. Then $A|_{\got g'}=0$. We will prove that $A|_{\got c}=0$. Recall that from Lemma \ref{lemma:Aderivacion} $A(\got c)\subseteq \got c$. 
    
    Fix $Z\in\got c$. By Proposition \ref{prop:ker_nodeg}, $\ker(\pi(Z))$ is non-degenerate and one obtains an orthogonal decomposition
    $\got v=\ker\pi(Z)\oplus(\ker\pi(Z))^\perp$. 
    
    From \eqref{eq:H_aut_lie}, $[B,\pi(Z)]=\pi(A(Z))$. So, if $X,Y\in\ker(\pi(Z))$, since $\pi(Z)\in \got{so}(\got v)$, 
    \begin{align*}
    \langle\pi(A(Z))(X),Y\rangle
        &=\langle B\pi(Z)(X)-\pi(Z)B(X),Y\rangle\\
        &=\langle B(X),\pi(Z)Y\rangle=0.
    \end{align*}
    Thus $\pi(AZ)X\in(\ker\pi(Z))^\perp$ for each $X\in \ker\pi(Z)$. But  $[\pi(AZ),\pi(Z)]=\pi([AZ,Z])=0$, so $\pi(AZ)$ also preserves $\ker\pi(Z)$. Hence $\pi(AZ)X=0$ for each $X\in \ker\pi(Z)$, i.e., 
    \begin{equation}\label{eq:auxpiaz}
        \pi(AZ)|_{\ker\pi(Z)} = 0.
    \end{equation}
    
    It remains to consider the action of $\pi(AZ)$ on $(\ker\pi(Z))^\perp$. From Proposition \ref{prop:ker_nodeg}
    \begin{equation*}
        (\ker\pi(Z))^\perp=\bigoplus_{i\in I}\got v_i, \quad I=\{i:\lambda_i(Z)\neq0\},
    \end{equation*}
    and from Theorems \ref{lema:tec} and \ref{teo:imp1},  on $\got v_i$ such that $\lambda_i(Z)\neq 0$,
    \begin{equation*}
        \pi(Z)|_{\got v_i}=\lambda_i(Z)J_i, \quad \pi(AZ)|_{\got v_i}=\lambda_i(AZ)J_i
    \end{equation*}
    with $J_i$ invertible. Fix $i\in I$ and put
    \begin{equation*}
        K_i=\lambda_i(Z)J_i, \quad \alpha_i=\frac{\lambda_i(AZ)}{\lambda_i(Z)},
    \end{equation*}
    and let $B_i$ be the orthogonal projection of $B|_{\got v_i}$ onto $\got v_i$. Observe that $\pi(AZ)|_{\got v_i}=\alpha_iK_i$. So equation \eqref{eq:H_aut_lie} gives
    \begin{equation*}
        B_iK_i-K_iB_i=\alpha_iK_i.
    \end{equation*}
    Since $J_i$ is invertible and $\lambda_i(Z)\neq 0$, $K_i$ is invertible. Therefore
    \begin{equation*}
        K_i^{-1}B_iK_i-B_i = \alpha_i\Id.
    \end{equation*}
    Now, $\trace(K_i^{-1}B_iK_i-B_i)=0$, and so we obtain $\alpha_i=0$. Hence $\pi(AZ)_{|\got v_i}=0$ if $\lambda_i(Z)\neq 0$. Therefore $\pi(AZ)$ vanishes on $(\ker\pi(Z))^\perp$. This together with \eqref{eq:auxpiaz} gives $\pi(AZ)=0$. Since  $\pi$ is faithful, we get $AZ=0$. Thus $A=0$, and the equation defining $(A,B)\in \got u$ in Proposition \ref{prop:h_aut_general} reduces to
    $[B,\pi(Z)]=0$ for all $Z\in\got g$. 
    
    The converse is immediate, and the last assertion follows from Proposition \ref{prop:h_aut_general}.
\end{proof}

We turn now to degenerate Lorentzian data sets.

\begin{lemma}\label{lem:null-line-quotient}
    Let $(\got g,\got v,\pi)$ be a degenerate Lorentzian data set and let $0\neq E\in \got v$ be a lightlike vector such that $\bigcap_{Z\in \got g}\ker(\pi(Z))=\R E$. Then:
    \begin{enumerate}
        \item For every $(A,B)\in\got h^{\aut}$,
        $B(\R E)\subset\R E$ and $B((\R E)^\perp)\subset (\R E)^\perp.$
        \item The quotient $\overline{\got v}=(\R E)^\perp/\R E$
        is Riemannian, and $\pi$ induces a representation $\bar\pi:\got g\to\got{so}(\overline{\got v})$.
        
        \item Every $B$ as above induces $\bar B\in\got{so}(\overline{\got v})$, and $[\bar B,\bar\pi(Z)]=\bar\pi(AZ) \quad\forall Z\in\got g.$
    \end{enumerate}
\end{lemma}

\begin{proof}
    Since $E$ is lightlike and $\got v$ is Lorentzian, the restriction of the metric to $(\R E)^\perp$ is positive semidefinite
    and $\operatorname{rad}((\R E)^\perp) = (\R E)^\perp\cap((\R E)^\perp)^\perp = \R E.$
    Consequently, the metric induces a positive definite inner product on $\overline{\got v}=(\R E)^\perp/\R E$,
    given by
    \begin{equation*}
        \langle X+\R E,Y+\R E\rangle_{\overline{\got v}} = \langle X,Y\rangle_{\got v}.
    \end{equation*}
    Thus $\overline{\got v}$ is Riemannian.
    Since $\pi(Z)(E)=0$ for all $Z\in\got g$, equation \eqref{eq:H_aut_lie} applied to $E$ gives
    \begin{equation*}
        0=\pi(AZ)E=[B,\pi(Z)]E=-\pi(Z)BE.
    \end{equation*}
    Thus $BE\in\R E$. Since $B$ is skew-symmetric, $B((\R E)^\perp)\subset (\R E)^\perp$. The remaining assertions follow by passing to the quotient.
\end{proof}

\begin{lemma}\label{lema:kerk}
    Let $(\got g,\got v,\pi)$ be a degenerate Lorentzian data set and let $0\neq E\in\got v$ be a lightlike vector such that $\bigcap_{Z\in\got g}\ker\pi(Z)=\mathbb RE$. Set $\overline{\got v}=(\mathbb RE)^\perp/\mathbb RE$ and let $\bar{\pi}:\got g \to \mathfrak{so}(\overline{\got v})$ be the induced representation given by Lemma \ref{lem:null-line-quotient}. Put $\got k=\ker\overline{\pi}$. Choose a lightlike vector $F\in\got v$ such that $\langle E,F\rangle=1,$ and set
    \begin{equation*}
        \got b=(\mathbb RE\oplus\mathbb RF)^\perp.
    \end{equation*}
    Then, for every $Z\in\got k$, there exists a unique vector $u_Z\in\got b$ such that
    \begin{equation*}
        \begin{cases}
            \pi(Z)E=0, & \\
            \pi(Z)X=-\langle u_Z,X\rangle E, &  X\in\got b,\\
            \pi(Z)F=u_Z.&
        \end{cases}
    \end{equation*}
    Moreover, the map $\varphi: \got k\to\got b,$ $\varphi(Z)=u_Z$ is linear and injective. In particular,
    \begin{equation*}
        [\pi(Z),\pi(T)] = 0 \quad \text{for all }Z,T\in\got k,
    \end{equation*}
    and hence $\got k$ is abelian.
\end{lemma}

\begin{proof}
    Since $\got v =\mathbb RE\oplus\got b\oplus\mathbb RF$, hence $(\mathbb RE)^\perp=\mathbb RE\oplus\got b$. Let $Z\in\got k$ and $X\in (\R E)^{\perp}$. Then $\bar\pi(Z)=0$ and so  $\mathbb RE=\bar 0=\bar\pi(Z)(X+\R E)=\pi(Z)X+\mathbb RE$. So $\pi(Z)((\R E)^{\perp})\subseteq \R E$ for each $Z\in \got k$.
    
    Hence, there exists a linear functional $\lambda_Z:\got b\to \mathbb \R$ such that for every $X\in\got b$, one has 
    \begin{equation*}
        \pi(Z)X=\lambda_Z(X) E.
    \end{equation*}
    Since $\got b$ is Riemannian, there exists a unique $X_Z\in \got b$ such that 
    \begin{equation*}
        \lambda_Z(X)=\langle X_Z,X\rangle.
    \end{equation*}
    Put now $\pi(Z)(F)=\alpha E+\beta F+u_Z$, with $u_Z\in \got b$. Since $\pi(Z)$ is skew-symmetric, one gets $\alpha=\langle \pi(Z)F,F\rangle=0$ and $\beta=\langle \pi(Z)F,E\rangle=0$. So $\pi(Z)F=u_Z\in \got b$. Moreover,  
    \begin{equation*}
        \lambda_Z(X)=\langle \pi(Z)X,F\rangle=-\langle X,u_Z\rangle
    \end{equation*}
    and so $X_Z=-u_Z$. 
    
    The map $\varphi:\got k\to \got b$, $\varphi(Z)=u_Z$ is linear, since $u_Z=\pi(Z)F$. It is also injective. Indeed, if $u_Z=0$, then the formulas above give $\pi(Z)(Y)=0$ for each $Y\in \got v$. Hence $\pi(Z)=0$ on $\got v$. Since $\pi$ is faithful, it follows that $Z=0$.
    
    Finally, let $Z,T\in\got k$. For $X\in\got b$, we have
    \begin{equation*}
        \pi(Z)\pi(T)X = -\langle u_T,X\rangle\pi(Z)E = 0,
    \end{equation*}
    and similarly $\pi(T)\pi(Z)X=0.$ So $[\pi(Z),\pi(T)]|_{\got b}=0$. Trivially $[\pi(Z),\pi(T)]E=0$. 
    
    Finally, on $F$ we have
    \begin{align*}
        \pi(Z)\pi(T)F &=\pi(Z)u_T= -\langle u_Z,u_T\rangle E, \\ 
        \pi(T)\pi(Z)F & = \pi(T)u_Z = -\langle u_T,u_Z\rangle E.
    \end{align*}
    We conclude that $[\pi(Z),\pi(T)]=0$. Since $\pi$ is a faithful representation, $ [Z,T]=0$ for all $Z,T\in \got k$.  Thus $\got k$ is abelian.
\end{proof}

\begin{theorem}\label{thm:degenerate-isotropy}
    Let $(\got g,\got v,\pi)$ be a degenerate Lorentzian data set, $0\neq E\in \got v$ a lightlike vector such that $\bigcap_{Z\in \got g}\ker(\pi(Z))=\R E$ and $\overline{\got v}=(\R E)^{\perp}/\R E$. Let $\bar\pi:\got g\to \got{so}(\overline{\got v})$ be the representation given in Lemma \ref{lem:null-line-quotient} and set $\got k=\ker(\bar\pi)$. Then:
    \begin{enumerate}
        \item $\got k$ is an abelian central ideal: $\got k\subset \got z(\got g)=\got c.$
        \item If $\got h^{\aut}=\got g'\oplus\got u$ and $(A,B)\in\got u$, then
        \begin{equation*}
            A(\got g)\subset\got k,
            \quad
            A|_{\got k^\perp}=0,
            \quad
            A|_{\got k}\in\got{so}(\got k).
        \end{equation*}
        \item The induced operator $\bar B$ commutes with the quotient representation, i.e., $[\bar B,\bar\pi(Z)]=0$ for all $ Z\in\got g.$
        \item With respect to the decomposition $\got g=\got k\oplus\got k^\perp$,
        $(A,B)\in \got u$ if and only if $(A,B)\in \got{so}(\got k)\oplus \got{so}(\got v)$ such that 
        \begin{equation*}
            \begin{cases}
                [B,\pi(Z)]=\pi(A(Z)),&Z\in\got k,\\
                [B,\pi(Z)]=0,&Z\in\got k^\perp
            \end{cases}.
        \end{equation*}
    In particular, if $\got k=\{0\}$, then $\got u=\End_{\pi}(\got v)\cap\got{so}(\got v)$.
\item For each $(A,B)\in \got h^{\aut}$, $B(E)=0$. 
        \end{enumerate}
\end{theorem}

\begin{proof}
    $\got k$ is an ideal of $\got g$, since it is the kernel of a Lie algebra homomorphism. From Lemma \ref{lema:kerk} $\got k$ is abelian.
    Since $\got g$ is compact with a Riemannian ad-invariant metric, $\got k^\perp$ is also an ideal of $\got g$. Hence $[\got k,\got k^\perp]\subset\got k\cap\got k^\perp=\{0\}$,
    and since $\got k$ is abelian, one gets that  $\got k\subset\got c$.
    
    Since  $\got v$ is Riemannian, we can apply Lemma \ref{lemma:central-action} to the representation $\bar\pi:\got g\to \got{so}(\overline{\got v})$, and obtain a decomposition 
    \begin{equation*}
        \overline{\got v}=\overline{\got w}_0\oplus\overline{\got v}_1\oplus\cdots\oplus\overline{\got v}_r,
    \end{equation*}
    where $\got c$ acts trivially on $\overline{\got w}_0$, and for each $W\in \got c$, $\bar\pi(W)|_{\overline{\got v}_i}=\lambda_i(W)J_i$ for some $\lambda_i: \got c \to \R$ and $J_i^2=-\Id.$
    
    Let now $(A,B)\in\got u$. From Lemma \ref{lemma:Aderivacion}, $A$ is a derivation of $\got g$. From Proposition \ref{prop:h_aut_general}, $A|_{\got g'}=0$ and $A(\got c)\subset\got c$. By Lemma~\ref{lem:null-line-quotient},
    \begin{equation}\label{eq:auxbarB}
        [\bar B,\bar\pi(Z)]=\bar\pi(AZ)
    \end{equation}
    where $\bar B\in \got{so}(\overline{\got v})$ is induced to the quotient by $B$. We will show that $\bar\pi(AZ)=0$ for every $Z\in\got c$.
    
    So fix $Z\in\got c$. Observe first that $\bar\pi(AZ)|_{\overline{\got w}_0}=0$. Fix $i=1,\ldots,r$ and let $B_i$ be the projection of $\bar B|_{\overline{\got v}_i}$ onto $\overline{\got v}_i$. Projecting  \eqref{eq:auxbarB} onto $\overline{\got v}_i$, we get
    \begin{equation*}
        \lambda_i(Z)(B_iJ_i-J_iB_i) = \lambda_i(AZ)J_i.
    \end{equation*}
    Multiplying both sides by $J_i^{-1}$, we obtain in the left-hand side a skew-symmetric operator, and on the right-hand  side $\lambda_i(AZ)\Id$. So, taking traces, we get that $\lambda_i(AZ)=0$. Since $i$ is arbitrary, we conclude that $\bar\pi(AZ)=0$.
    
    Thus $AZ\in\got k$ for every $Z\in\got c$, and since $A|_{\got g'}=0$, we get that $A(\got g)\subset\got k.$
    If $X\in\got k^\perp$, then for every $Y\in\got g$,
    \begin{equation*}
        \langle AX,Y\rangle = -\langle X,AY\rangle=0,
    \end{equation*}
    so $AX=0$. Hence $A|_{\got k^\perp} = 0$ and $A|_{\got k} \in \got{so}(\got k)$. The final description of $\got u$ is then just the description in Proposition \ref{prop:h_aut_general} written with respect to $\got g=\got k\oplus\got k^\perp$.

    Let us finally prove the last point. In the decomposition $\got h^{\aut}=\got g'\oplus \got u$, the elements of $\got g'$ act on $\got v$ by $\pi(Z)$, and hence annihilate $E$. Therefore it is enough to prove the result for $(A,B)\in\got u$.

    By Lemma~\ref{lem:null-line-quotient}, there exists $\lambda\in\mathbb R$ such that $B(E)=\lambda E$. As in Lemma \ref{lema:kerk}, choose a lightlike vector $F\in\got v$ such that $\langle E,F\rangle=1$, and set $\got b=(\mathbb RE\oplus\mathbb RF)^\perp.$ 

    Decompose $B(F)=\alpha E+\beta F+b$, with $b\in \got b$. Since $B$ is skew-symmetric, $\alpha=\langle B(F),F\rangle=0$, $\beta=\langle B(F),E\rangle =-\lambda$. So $B(F)=b-\lambda F$. Now, if $X\in \got b$, with the same argument as before one gets that  $B(X)=X_{\got b}-\langle b,X\rangle E$ for some $X_{\got b}\in \got b$. Hence,  one can consider the endomorphism $\tilde B:\got b\to \got b$, given by $$\tilde B(X)=X_{\got b}=B(X)+\langle b,X\rangle E.$$
    Since $B$ is skew-symmetric, it is not hard to see that so is $\tilde B$. So $\tilde B\in \got{so}(\got b)$. 
    
    By Lemma~\ref{lema:kerk}, for every $Z\in\got k$ there exists $u_Z\in\got b$ such that
    \begin{equation*}
        \pi(Z)F=u_Z,
        \qquad
        \pi(Z)X=-\langle u_Z,X\rangle E
        \quad\text{for every }X\in\got b,
    \end{equation*}
    and the map $\varphi:\got k\to\got b$, $\varphi(Z)=u_Z$  is linear and injective.

    Let $Z\in\got k$. Applying \eqref{eq:H_aut_lie} to $F$, we obtain
    \begin{align*}
        u_{AZ}
        &=\pi(AZ)F=B(u_Z)-\pi(Z)(B(F))\\
        &=\widetilde B u_Z-\langle b,u_Z\rangle E
          -\pi(Z)(b-\lambda F)\\
        &=\widetilde B u_Z-\langle b,u_Z\rangle E
          +\langle u_Z,b\rangle E+\lambda u_Z
        =(\widetilde B+\lambda\Id)u_Z.
    \end{align*}
    Set $U=\varphi(\got k)\subset\got b$, then $\widetilde B(U)\subset U$, and
    \begin{equation*}
        \varphi\circ A|_{\got k}
        =
        (\widetilde B|_U+\lambda\Id)\circ\varphi.
    \end{equation*}
    Since $\widetilde B\in\got{so}(\got b)$ and $U$ is $\widetilde B$-invariant, $\widetilde B|_U\in\got{so}(U).$

    Taking the usual inner product $\langle C,D\rangle=\trace(A^*B)$ in $\End(U)$,  we get
    \begin{equation*}
        \langle\varphi\circ A,\varphi\rangle
        =
        \langle\widetilde B|_U\circ\varphi,\varphi\rangle
        +\lambda\langle\varphi,\varphi\rangle.
    \end{equation*}
    Since $A|_{\got k}\in\got{so}(\got k)$ and $\widetilde B|_U\in\got{so}(U)$, the first two inner products vanish. Hence $\lambda\langle\varphi,\varphi\rangle=0$. Since $\varphi$ is injective, $\varphi\neq0$, and therefore $\lambda=0$. Thus $B(E)=0$.
\end{proof}

\bibliographystyle{amsalpha}

\begin{thebibliography}{CdBR24}
	
	\bibitem[CdBR24]{ContiDelBarcoRossi2024}
	D.~Conti, V.~del Barco, and F.~Rossi, \emph{On uniqueness of ad-invariant
		metrics}, Tohoku Math. J. \textbf{76} (2024), no.~3, 317--359.
	
	\bibitem[CNN23]{ChenNikolayevskyNikonorov2022}
	Z.~Chen, Y.~Nikolayevsky, and Y.~Nikonorov, \emph{Compact geodesic orbit spaces
		with a simple isotropy group}, Ann. Global Anal. Geom. \textbf{63} (2023),
	no.~1, 7.
	
	\bibitem[CP09]{CorderoParker2009}
	L.~A. Cordero and P.~E. Parker, \emph{Isometry groups of pseudoriemannian
		$2$-step nilpotent lie groups}, Houston J. Math. \textbf{35} (2009), 49--72.
	
	\bibitem[CWZ22]{ChenWolfZhang2020}
	Z.~Chen, J.~A. Wolf, and S.~Zhang, \emph{{On the geodesic orbit property for
			Lorentz manifolds}}, J. Geom. Anal. \textbf{32} (2022), no.~3, 1--14.
	
	\bibitem[dBO14]{DelBarcoOvando2014}
	V.~del Barco and G.~P. Ovando, \emph{{Isometric actions on pseudo-Riemannian
			nilmanifolds}}, Ann. Global Anal. Geom. \textbf{45} (2014), no.~2, 95--110.
	
	\bibitem[DSO01]{DiScalaOlmos2001}
	A.~J. Di~Scala and C.~Olmos, \emph{{The geometry of homogeneous submanifolds of
			hyperbolic space}}, Math. Z. \textbf{237} (2001), 199--209.
	
	\bibitem[Ebe94]{Eberlein1994}
	P.~Eberlein, \emph{{Geometry of $2$-step nilpotent groups with a left invariant
			metric}}, Ann. Sci. École Norm. Sup. \textbf{27} (1994), no.~5, 611--660.
	
	\bibitem[Ebe96]{Eberlein1996}
	\bysame, \emph{Geometry of nonpositively curved manifolds}, Chicago Lectures in
	Mathematics, Chicago Press, 1996.
	
	\bibitem[Gor85]{Gordon1985}
	C.~S. Gordon, \emph{Naturally reductive homogeneous {Riemannian} manifolds},
	Can. J. Math. \textbf{37} (1985), no.~3, 467--487.
	
	\bibitem[Gue03]{Guediri2003}
	M.~Guediri, \emph{Lorentz geometry of $2$-step nilpotent lie groups}, Geom.
	Dedicata \textbf{100} (2003), no.~1, 11--51.
	
	\bibitem[JSC10]{JavaloyesSanchez2010}
	M.~A. Javaloyes and M.~Sánchez~Caja, \emph{{An introduction to Lorentzian
			geometry and its applications}}, XVI Escola de Geometria Diferencial, IME,
	USP, 2010.
	
	\bibitem[Kap81]{Kaplan1981}
	A.~Kaplan, \emph{{Riemannian nilmanifolds attached to Clifford modules}}, Geom.
	Dedicata \textbf{11} (1981), no.~2, 127--136.
	
	\bibitem[Lau98]{Lauret1998}
	J.~Lauret, \emph{{Naturally reductive homogeneous structures on $2$-step
			nilpotent Lie groups}}, Rev. Unión Mat. Argent. \textbf{41} (1998), no.~2,
	15--24.
	
	\bibitem[Lau99]{Lauret1999}
	\bysame, \emph{{Homogeneous nilmanifolds attached to representations of compact
			Lie groups}}, Manuscripta Math. \textbf{99} (1999), 287--309.
	
	\bibitem[NW23]{NikolayevskyWolf2022}
	Y.~Nikolayevsky and J.~A. Wolf, \emph{{The structure of geodesic orbit Lorentz
			nilmanifolds}}, J. Geom. Anal. \textbf{33} (2023), no.~3, 1--12.
	
	\bibitem[Ova13]{Ovando2013}
	G.~P. Ovando, \emph{{Naturally reductive pseudo-Riemannian $2$-step nilpotent
			Lie groups}}, Houston J. Math. \textbf{39} (2013), no.~1, 147--167.
	
	\bibitem[Wil82]{Wilson1982}
	E.~N. Wilson, \emph{{Isometry groups on homogeneous nilmanifolds}}, Geom.
	Dedicata \textbf{12} (1982), no.~3, 337--346.
	
	\bibitem[Wol62]{Wolf1962}
	J.~A. Wolf, \emph{{On locally symmetric spaces of non-negative curvature and
			certain other locally homogeneous spaces}}, Comment. Math. Helv. \textbf{37}
	(1962), no.~1, 266--295.
	
	\bibitem[Wu64]{Wu1964}
	H.~Wu, \emph{{On the de Rham decomposition theorem}}, Illinois J. Math.
	\textbf{8} (1964), no.~4, 291--311.
	
\end{thebibliography}

\end{document}